\documentclass[sn-mathphys,Numbered]{sn-jnl}% Math and Physical Sciences Reference Style
\usepackage{graphicx}%
\usepackage{multirow}%
\usepackage{amsmath,amssymb,amsfonts}%
\usepackage{amsthm}%
\usepackage{mathrsfs}%
\usepackage[title]{appendix}%
\usepackage{xcolor}%
\usepackage{textcomp}%
\usepackage{manyfoot}%
\usepackage{booktabs}%
\usepackage{algorithm}%
\usepackage{algorithmicx}%
\usepackage{algpseudocode}%
\usepackage{listings}%
\usepackage{comment}%
\newtheorem{theorem}{Theorem}[section]
\newtheorem{proposition}{Proposition}[section]
\newtheorem{lemma}{Lemma}[section]
\newtheorem{corollary}{Corollary}[section]

\newtheorem{remark}{Remark}[section]

\renewcommand{\proofname}{Proof}

\newcommand{\Complex}{ \mathbb{C} }

\newcommand{\Fourier}{ \mathcal{F}}
\newcommand{\FourierInverse}{ \mathcal{F}^{ - 1 } }

\newcommand{\Integer}{ \mathbb{Z} }

\newcommand{\Real}{ \mathbb{R} }

\newcommand{\Torus}{ \mathbb{T} }
\begin{document}

\title[Article]{Well-Posedness of the Cauchy Problem for One-Dimensional Nonlinear Diffusion Equations with Dynamic and Fourth-Type Boundary Conditions in the $L^p$–$L^q$ Maximal Regularity Setting}

%%=============================================================%%
%% Prefix	-> \pfx{Dr}
%% GivenName	-> \fnm{Joergen W.}
%% Particle	-> \spfx{van der} -> surname prefix
%% FamilyName	-> \sur{Ploeg}
%% Suffix	-> \sfx{IV}
%% NatureName	-> \tanm{Poet Laureate} -> Title after name
%% Degrees	-> \dgr{MSc, PhD}
%% \author*[1,2]{\pfx{Dr} \fnm{Joergen W.} \spfx{van der} \sur{Ploeg} \sfx{IV} \tanm{Poet Laureate} 
%%                 \dgr{MSc, PhD}}\email{iauthor@gmail.com}
%%=============================================================%%

%Authors
\author[1]{\fnm{Ken} \sur{Furukawa}}\email{furukawa@sci.u-toyama.ac.jp}
%\equalcont{These authors contributed equally to this work.}
%\author[2]{\fnm{Second} \sur{Author}}\email{iiauthor@gmail.com}
%\equalcont{These authors contributed equally to this work.}
%\author[3]{\fnm{Third} \sur{Author}}\email{iiiauthor@gmail.com}
%\equalcont{These authors contributed equally to this work.}

%Affiliations
\affil*[1]{\orgdiv{Faculty of Science, Academic Assembly}, \orgname{University of Toyama}, \orgaddress{\street{Gofuku 3190}, \city{Toyama}, \state{Toyama}, \postcode{930-8555}, \country{JAPAN}}}
%\affil[2]{\orgdiv{Department}, \orgname{Organization}, \orgaddress{\street{Street}, \city{City}, \postcode{10587}, \state{State}, \country{Country}}}
%\affil[3]{\orgdiv{Department}, \orgname{Organization}, \orgaddress{\street{Street}, \city{City}, \postcode{610101}, \state{State}, \country{Country}}}

%%==================================%%
%% sample for unstructured abstract %%
%%==================================%%
\abstract{
This paper addresses the local well-posedness of the Cauchy problem for a one-dimensional diffusion equation equipped with a dynamic boundary condition and an additional boundary condition that renders the one-dimensional Laplace operator self-adjoint.
The equation serves as a model for describing filtration in aquaria, originally introduced by the author and Kitahata.
The boundary condition treated in this work differs from classical types such as Dirichlet, Neumann, and Robin conditions; we refer to it as the fourth or FK-type boundary condition.
The boundary condition is designed to capture interactions between the two boundaries in the context of filtration.
The framework for establishing well-posedness is based on $L^p$-$L^q$ maximal regularity classes. 
The results of this paper are applicable to equations with a wide range of nonlinearities, including those of the Lotka–Volterra type.
}
%%\pacs[JEL Classification]{D8, H51}
%%\pacs[MSC Classification]{35A01, 65L10, 65L12, 65L20, 65L70}
\maketitle
%\backmatter
%\bmhead{Supplementary information}
%\bmhead{Acknowledgments}
%\section*{Declarations}
%\noindent

\section{Introduction}
This paper addresses the well-posedness of the Cauchy problem for the reaction-diffusion equations with the dynamic boundary condition
\begin{equation} \label{eq_filter_clogging}
    \begin{split}
        \begin{aligned}[t]
            &\partial_t v - \partial_x^2 v
            = f_v(v, \sigma) + g_v
            & x \in I,
            & \, t >0,\\
            &B_1(\theta(\sigma); v)
            = 0, \quad 
            B_2(\theta(\sigma); v)
            =0,
            & x \in \partial I,
            & \, t >0,\\
            &\frac{d\sigma}{dt}
            = f_\sigma(v, \sigma) + g_\sigma,
            & 
            & \, t >0,\\
        \end{aligned}
    \end{split}
\end{equation}
The domain is the interval $I = (-1,1)$ and the boundary is the left and right edges $\partial I = \{ x = \pm 1\}$.
The functions in the domain $v : [0, T) \times I \rightarrow \Real$ and the function on the boundary $\sigma : [0, T) \rightarrow \Real$ are unknowns.
We identify the two edges, but not periodic, so $\sigma$ can be a function defined only on the right edge.
The functions $v_0(x)$ and $\sigma_0$ are initial data of $v$ and $\sigma$, respectively.
The functions $f_v$ and $f_\sigma$ are nonlinear.
A typical case is
\begin{align} \label{eq_example_f_v_f_sigma}
    \begin{split}
        f_v(v, \sigma)
        & = - \theta(\sigma) \partial_x v
        + v (1 - v),\\
        f_\sigma(v, \sigma)
        & = \sigma (1 - \sigma)
        + \theta(\sigma)^2 \gamma_+ v.
    \end{split}
\end{align}
Compare with the nonlinear functions in \cite{FurukawaKitahata2024}.
Let $\gamma_\pm \varphi = \varphi(\pm 1)$ be the boundary trace operator.
The boundary conditions are defined with $B_1$ and $B_2$ are boundary operators defined by
\begin{align} \label{eq_BC}
    \begin{split}
        B_1(\theta; \varphi)
        &= (1 - \theta) \gamma_+ \varphi - \gamma_- \varphi,\\
        B_2(\theta; \varphi)
        &= \gamma_+ \partial_x \varphi - (1 - \theta) \gamma_- \partial_x \varphi,
    \end{split}
\end{align}
for some $\theta > 0$, which can be a function of $t \geq 0$.
The functions $g_v$ and $g_\sigma$ are forcing terms independent of $v$ and $\sigma$.
The system (\ref{eq_filter_clogging}) is a mathematical model introduced by the author and Kitahata \cite{FurukawaKitahata2024} to describe the filtration of liquids, \textit{e.g.}, water or air.
We refer to this boundary condition as the fourth boundary condition, or the Furukawa–Kitahata (FK) type boundary condition.
A typical example is the filtration system in an aquarium housing live fish.  
In this model, the boundary conditions replicate the function of a filtration device installed at the boundary.  
Some portion of microscopic particles (\textit{e.g.}, fish metabolites) in the aquarium passes through the filtration device at the boundary, driven by a pump.  
At this time, a fraction $\theta$ of the particles is absorbed by the filtration device, while the remaining particles are returned to the aquarium from the opposite side.  
This process helps to maintain the cleanliness of the water.

The second equation in (\ref{eq_BC}) serves a mathematical requirement to make the Laplace operator $- \partial_x^2$ associated with the domain
\begin{align} \label{eq_domain}
    D(-\partial_x^2)
    = \{
        \varphi \in H^{2, q}(I)
        \,:\,
          B_1(\theta; \varphi)
        = B_2(\theta; \varphi)
        = 0
    \}
\end{align}
as non-negative and self-adjoint when $q = 2$.
It can be seen from the formula
\begin{align*}
    -\int_I 
        \partial_x^2 \varphi(x) \varphi(x)
    dx
    = \int_I 
        \vert
            \partial_x \varphi(x)
        \vert^2
    dx
    \geq 0.
\end{align*}
See also Lemma 3.2.1 in the book by Sohr \cite{Sohr2001}.
Therefore, it can be understood that the boundary conditions (\ref{eq_BC}) represent intermediate boundary conditions from the periodic boundary condition ($\theta=0$) to the Dirichlet-Neumann boundary condition ($\theta=1$) so that the Laplace operator $-\partial_x^2$ is non-negative self-adjoint.
See also the book by Schmudgen \cite{Schmudgen2012} for the complete characterization of boundary conditions leading the self-adjointness of the one-dimensional Laplace operator $- \partial_x^2$.
The positivity of the spectrum implies the decay of the semigroup generated by the one-dimensional Laplace operator with the FK-type boundary condition.  
This decay corresponds to the effect of filtration.  
At the endpoint cases $\theta = 0$ and $\theta = 1$, the boundary condition reduces to the periodic boundary condition and the Dirichlet–Neumann mixed boundary condition, respectively.  
The case $\theta = 0$ represents the absence of filtration, and thus no decay occurs.  
In contrast, the case $\theta = 1$ corresponds to maximal filtration, leading to the fastest decay.
As the concentration of fish metabolites $\sigma$ increases, the filtration ratio $\theta(\sigma)$ tends to decrease.  
Consequently, the smallest eigenvalue of the operator $- \partial_x^2$ approaches zero, and the filtration effect diminishes accordingly.

This paper addresses the well-posedness of the Cauchy problem for the parabolic equation~(\ref{eq_filter_clogging}), hereafter referred to simply as the well-posedness of~(\ref{eq_filter_clogging}).
There are many results about the maximal regularity for the abstract parabolic type equation
\begin{align} \label{eq_abstract_parabolic_eq}
    \begin{split}
        \frac{du}{dt} + A u
        & = f \quad \text{in $X$ for $t > 0$}, \\
        u(0)
        & = u_0,
    \end{split}
\end{align}
in a UMD Banach space $X$ where $f: (0, T) \rightarrow X$ is the forcing and $u_0$ is the initial data in a trace space $X_\gamma$.
The operator $A$ is a densely defined closed operator and is a generator of an analytic semigroup.
For the theory of an analytic semigroup of a closed operator, the reader is referred to \cite{Kato1966}.
The maximal $L^p$-$L^q$ regularity for $p, q \in (1, \infty)$ means that $\frac{du}{dt}, Au \in L^p(0,T; X) \hookrightarrow C([0, T); X_\gamma)$ for a Banach space $X$ of a closed subspace of $L^q$ and the inequality
\begin{align*}
    \left \Vert
        \frac{du}{dt}
    \right \Vert_{L^p(0,T; X)}
    + \Vert
        A u
    \Vert_{L^p(0,T; X)}
    \leq C (
        \Vert
            u_0
        \Vert_{X_\gamma}
        + \Vert
            f
        \Vert_{L^p(0,T; X)}
    )
\end{align*}
holds.
Dore and Venni \cite{DoreVenni1987} proved that the fractional power of $A$ is a sufficient condition; see also Dore \cite{Dore1993} for an application to parabolic partial differential equations (PDEs), and Giga and Sohr \cite{GigaSohr1991} for an application to the Stokes operator.
The theory of operators with fraction imaginary power leads to the theory of operator with bounded $H^\infty$-calculus, see for example \cite{AuscherMcIntoshNahmod1997, CowlingDoustMcIntoshYagi1996, McIntosh1986}.
The necessary and sufficient condition for maximal regularity was characterized by Weis \cite{Weis2001}.
The $\mathcal{R}$-boundedness of the semigroup is a necessary and sufficient condition of the maximal regularity.
The functional analytic theory of maximal regularity including a sum of closed operators has been developed by many researchers.
The reader is referred to \cite{KunstmannWeis2004, DenkPrussZacher2008, PrussSimonett2004, PrussSimonett2016} for comprehensive studies on the subject.

The one-dimensional Laplace operator $\partial_x^2 =: \mathcal{A}(t)$ in the equation (\ref{eq_filter_clogging}) is time-dependent because the domain depends on $\theta = \theta(\sigma(t))$.
The solution $u(t)$ to (\ref{eq_abstract_parabolic_eq}) when $A = \mathcal{A}(t)$ is constructed by the evolution operator $\mathcal{T}(t, s)$ and the Duhamel formula
\begin{align*}
    u(t)
    = \mathcal{T}(t, 0) u_0
    + \int_0^t
        \mathcal{T}(t, \tau) f(\tau)
    d\tau
\end{align*}
at least formally.
The construction of the evolution operator in this case is developed by Kato and Tanabe \cite{KatoTanabe1962}.
If the domain is a subspace of the Hilbert space and is time-independent, the evolution operator can be constructed by the corresponding quadratic-form from Lions's method \cite{Lions1959, Lions1968}.
See \cite{Tanabe1960, Tanabebook,Yagi1991, Yagibook}

There are several results concerning well-posedness of PDEs under dynamic boundary conditions.
Hintermann \cite{Hintermann1989} studied the well-posedness of various type of PDEs, for example elliptic, parabolic, and hyperbolic types, in Sobolev spaces.
Denk, Pr\"{u}ss, and Zacher \cite{DenkPrussZacher2008} have established the well-posedness of $2m$-th order parabolic equations for where $m \in \Integer_{\geq 1}$ under dynamic boundary conditions in the $L^p$-$L^p$ maximal regularity framework.
This is a highly general result that is applicable to linear parabolic equations under generalized Dirichlet, Neumann, and Robin boundary conditions.
The first author and Kajiwara \cite{FurukawaKajiwara2021} have proved sufficient conditions for well-posedness in the $L^p$-$L^q$ maximal regularity setting. 
The reader is referred to the book by Pr\"{u}ss and Simonett \cite{PrussSimonett2016} for comprehensive knowledge about the well-posedness of PDEs in maximal $L^p$-$L^q$ framework.
However, these results are not applicable to equation~(\ref{eq_filter_clogging}) due to the interaction between the two boundaries.
The nonlinearity of $\sigma$ in the system (\ref{eq_filter_clogging}) is embedded in the most essential part that determines the decay structure of the system.  
To the best of the author's knowledge, there are no existing mathematical analyses for such systems other than the previous work \cite{FurukawaKitahata2024}.
The existence theorem for the equations of filtration \cite{FurukawaKitahata2024} is based on the evolution operator theory in $L^2$ developed by Tanabe, as mentioned, together with the Leray–Schauder principle.  
Although the strategy is standard, the argument is technically involved, because Tanabe's theory requires $C^{1,\alpha}$-H\"{o}lder regularity for $\theta(\sigma(t))$.  
As a result, the solution space for $v$ must be $H^3$ because the equation of $\sigma$ includes the boundary trace of $v$, which is a significantly higher regularity class than that of typical second-order nonlinear parabolic equations.  
In this paper, we propose a new framework to establish well-posedness in the $L^p$–$L^q$ maximal regularity class.  
This framework allows for solutions in less regular spaces than those required in previous work, and yields results that ensure the existence of local-in-time well-posedness of equations with various types of nonlinearities.

The solvability of parabolic and elliptic equations under various boundary conditions—such as Dirichlet, Neumann, Robin, and their generalizations—as well as under dynamic boundary conditions, has been extensively studied.  
In our research, we propose a novel combination of boundary conditions.  
This idea interprets the combination of the boundary condition~(\ref{eq_BC}) and the dynamic boundary condition as a mechanism for controlling internal unknowns (filtration), thereby providing a new perspective on the effectiveness of formulating PDEs under dynamic boundary conditions.

To introduce the main result, we begin by setting notation.
Let $\theta_0 = \theta(\sigma_0)$ and $A_q(\theta_0) = \partial_x^2$ associated with the same domain as \eqref{eq_domain} for $\theta = \theta_0$.
We denote the real-interpolation space by
\begin{align*}
    D_{p, q, \eta} (\theta_0)
    = (L^q(I), D(A_q(\theta_0)))_{\eta, p}, \quad \eta \in (0,1).
\end{align*}
%When $\theta_0$ is clear, we simplify write as $A_q(\theta_0) = A_q$.
Since $A_q(\theta_0)$ admits a bounded $H^\infty$-calculus, as shown in Theorem~\ref{thm_H_infty_calculus_of_A}, this fact combined with the trace theorem allows us to characterize $D_{p, q, \eta}(\theta_0)$ as
\begin{equation*}
    D_{p, q, \eta} (\theta_0)
    = \left \{
        \begin{alignedat}{3}
            B^{2\eta}_{q, p}(I),
            & \quad &
            & \eta \in (0, 1/2p),\\
            \{
                \varphi \in B^{2\eta}_{q, p}(I)
                \, ; \,
                B_1(\theta_0; \varphi)
                = 0
            \},
            & \quad & 
            & \eta \in (1/2p, 1/2 + 1/2p),\\
            \{
                \varphi \in B^{2\eta}_{q, p}(I)
                \, ; \,
                B_1(\theta_0; \varphi)
                = B_2(\theta_0; \varphi)
                = 0
            \},
            & \quad &  
            & \eta \in (1/2 + 1/2p, 1).
        \end{alignedat}
    \right .
\end{equation*}
For the characterization of the initial trace class, the reader is referred to \cite{GigaGriesHieberHusseinKashiwabara2017_H_infty}, for instance.
Let
\begin{align*}
    E_{v, 0, p, q, T}
    = L^p(0,T; L^q(I)), \quad
    E_{\sigma, 0, p, T}
    = L^{p}(0,T), \quad
    E_{\sigma, 1, p, T}
    = H^{1, p}(0,T).
\end{align*}
The main result of this paper is
\begin{theorem} \label{thm_main}
    Let $1 < p, q < \infty$ and $l, m \in \Integer_{\geq 1}$.
    Let $v_0 \in D_{p, q, 1 - 1/p}(\theta_0)$ and $\sigma_0 \in \Real$.
    Let $g_v \in E_{v, 0, p, q, T}$ and $g_\sigma \in E_{\sigma, 0, p, T}$.
    Assume that $f_v(v, \sigma)$ and $f_\sigma(w, \sigma)$ satisfy the boundedness
    \begin{gather} \label{eq_assumption_1_of_f_in_main_thm}
        \begin{split}
            \Vert
                f_v(v, \sigma)
            \Vert_{E_{v, 0, p, q, T}}
            \leq C (
                \delta(T)
                + \nu^l
                + s^l
            )
            (
                \nu
                + s
            ),\\
            \Vert
                f_\sigma(v, \sigma)
            \Vert_{L^p(0, T)}
            \leq C (
                \delta(T)
                + \nu^m
                + s^m
            )
            (
                \nu
                + s
            ),\\
            (
                \nu
                = \Vert
                    v
                \Vert_{E_{v, 1, p, q, T}}, \,
                s = \Vert
                    \sigma
                \Vert_{E_{\sigma, 1, p, T}}
            )
        \end{split}
    \end{gather}
    and the locally Lipschitz continuity
    \begin{align} \label{eq_assumption_2_of_f_in_main_thm}
        \begin{split}
            & \Vert
                f_v(v_1, \sigma_1)
                - f_v(v_2, \sigma_2)
            \Vert_{E_{w, 0, p, q, T}}\\
            & \quad \leq C (
                \delta(T)
                + \nu_1^l
                + \nu_2^l
                + s_1^l
                + s_2^l
            )
            (
                \Vert
                    v_1
                    - v_2
                \Vert_{E_{w, 1, p, q, T}}
                + \Vert
                    \sigma_1
                    - \sigma_2
                \Vert_{_{\sigma, 1, p, T}}
            ),\\
            & \Vert
                f_\sigma(v_1, \sigma_1)
                - f_\sigma(v_2, \sigma_2)
            \Vert_{E_{w, 0, p, q, T}}\\
            & \quad \leq C (
                \delta(T)
                + \nu_1^m
                + \nu_2^m
                + s_1^m
                + s_2^m
            )
            (
                \Vert
                    v_1
                    - v_2
                \Vert_{E_{w, 1, p, q, T}}
                + \Vert
                    \sigma_1
                    - \sigma_2
                \Vert_{E_{\sigma, 1, p, T}}
            ),\\
            & \quad\quad\quad\quad\quad\quad
            (
                \nu_j
                = \Vert
                    v_j
                \Vert_{E_{w, 1, p, q, T}}, \,
                s_j
                = \Vert
                    \sigma_j
                \Vert_{E_{\sigma, 1, p, T}}
            )
        \end{split}
    \end{align}
    for some constant $\delta(T) > 0$, which tends to zero as $T \rightarrow 0$.
    Then there exists $T_0 \in (0, T)$ and a unique solution
    \begin{gather*}
        \sigma(t)
        \in \Real, \quad
        v(t)
        \in D_{p, q, 1 - 1/p}(\theta(\sigma(t))), \text{ continuously,}
    \end{gather*}
    to (\ref{eq_filter_clogging}) in the strong sense in $E_{v, 0, p, q, T} \times E_{\sigma, 0, p, T}$ satisfying
    \begin{align}
        \sigma
        \in H^{1, p}(0,T_0), \quad
        \partial_t v, \,
        A_q(\theta(\sigma)) v
        \in E_{v, 0, p, q, T}
    \end{align}
    and the inequality
    \begin{align*}
        & \Vert
            \partial_t v
        \Vert_{E_{w, 0, p, q, T_0}}
        + \Vert
            A_q(\theta(\sigma(t))) v
        \Vert_{E_{w, 0, p, q, T_0}}
        + \Vert
            \sigma
        \Vert_{E_{\sigma, 1, p, T}}\\
        & \leq C_{T} (
            \Vert
                w_0
            \Vert_{D_{p, q, 1 - 1 / p}}
            + \vert
                \sigma_0
            \vert
            + \Vert
                g_w
            \Vert_{E_{w, 0, p, q, T_0}}
            + \Vert
                g_\sigma
            \Vert_{E_{\sigma, 0, p, T}}
        ).
    \end{align*}
\end{theorem}

\begin{remark}
    When $f_v(v, \sigma)$ and $f_\sigma(v, \sigma)$ are given by (\ref{eq_example_f_v_f_sigma}), we find from the Sobolev inequality and the trace theorem that
    \begin{align*}
        & \Vert
            f_v(v, \sigma)
        \Vert_{L^p(0, T; L^q(I))}\\
        & \leq C \Vert
            v
        \Vert_{L^p(0, T; H^{1, q}(I))}
        + C \Vert
            v
        \Vert_{L^{2 p}(0, T; L^{2 q}(I))}^2\\
        & \leq C T^\eta (
            \Vert
                v
            \Vert_{
                H^{1, p}(0, T; L^q(I))
                \cap L^p(0, T; H^{2, q}(I))
            }
            + \Vert
                v
            \Vert_{
                H^{1, p}(0, T; L^q(I))
                \cap L^p(0, T; H^{2, q}(I))
            }^2
        ), \\
        & \Vert
            f_\sigma(v, \sigma)
        \Vert_{L^p(0, T)}\\
        & \leq C T^\eta (
            \Vert
                \sigma
            \Vert_{H^{1, p}(0, T)}
            + \Vert
                \sigma
            \Vert_{H^{1, p}(0, T)}^2
            + \Vert
                v
            \Vert_{
                H^{1, p}(0, T; L^q(I))
                \cap L^p(0, T; H^{2, q}(I))
            }
        )
    \end{align*}
    for all $p, q \in (1, \infty)$.
    In this case Theorem \ref{thm_main} is applicable for all $p, q \in (1, \infty)$.
\end{remark}

The strategy to prove Theorem~\ref{thm_main} is as follows.  
We begin by considering the linearized system~\ref{eq_filter_clogging} at time $t = 0$.
The boundary condition~(\ref{eq_BC}) is converted into an inhomogeneous one with a time-independent filtration ratio, given by  
\begin{align*}
    \begin{split}
        B_1(\theta_0; v)
        &= (\theta(\sigma) - \theta_0)\, \gamma_+ v,\\
        B_2(\theta_0; v)
        &= - (\theta(\sigma) - \theta_0)\, \gamma_- \partial_x v.
    \end{split}
\end{align*}
When considering local-in-time well-posedness, it is reasonable to consider the smallness of $(\theta(\sigma) - \theta_0)$.  
To eliminate the inhomogeneity in the boundary conditions, we introduce a correction function $v_c$ such that $\tilde{v} := v - v_c$ satisfies the homogeneous boundary conditions $B_j(\theta_0; \tilde{v}) = 0$ for $j = 1, 2$.  
The function $v_c$ depends linearly and continuously on $v$, and hence can be expressed as $\tilde{v} = G(\theta_0, \theta(\sigma)) v$ for some bounded operator $G$ acting on both $L^q(I)$ and the Sobolev space $H^m(I)$ for $m \in \mathbb{Z}_{\geq 1}$.  
The smallness of $(\theta(\sigma) - \theta_0)$ ensures the invertibility of $G(\theta_0, \theta(\sigma))$, so that $v$ can be recovered from $\tilde{v}$.  
Therefore, it is sufficient to establish the well-posedness of the system satisfied by $(\tilde{v}, \sigma)$.

The resulting system is a nonlinear parabolic equation.  
Its well-posedness in the $L^p$–$L^q$ setting is obtained by combining the maximal regularity of the one-dimensional Laplace operator $A$ associated with the boundary condition~(\ref{eq_BC}) for $\theta = \theta_0$, with the Banach fixed point theorem.  
The bounded $H^\infty$-calculus for $A$ follows from the explicit formula of the resolvent operator $(\lambda - A)^{-1}$ derived via Fourier analysis.  
This formula implies that for any analytic function $h$ defined on a sector $\Sigma$ and decaying as $|\lambda| \to 0$ and $|\lambda| \to \infty$, the operator $h(A)$ is bounded on $L^q(I)$.

Under the assumptions on the nonlinearities $f_v$ and $f_\sigma$, the Fujita–Kato method can be applied.  
The nonlinear terms $f_v(v, \sigma)$ and $f_\sigma(v, \sigma)$ are rewritten as $f_v(G(\theta_0, \theta(\sigma))^{-1} \tilde{v}, \sigma)$ and $f_\sigma(G(\theta_0, \theta(\sigma))^{-1} \tilde{v}, \sigma)$, respectively.  
Since $G(\theta_0, \theta(\sigma))^{-1}$ is continuous in the maximal $L^p$–$L^q$ class with $\sigma$-independent bounds, the transformed nonlinearities satisfy polynomial-type estimates with respect to $\tilde{v}$ and $\sigma$ as given in~(\ref{eq_assumption_1_of_f_in_main_thm})–(\ref{eq_assumption_2_of_f_in_main_thm}).  
Using this standard method, we can construct a unique local-in-time solution to~(\ref{eq_filter_clogging}).  

We introduce notation.  
For a Banach space $X$, we denote by $\Vert \cdot \Vert_{X \rightarrow X}$ the operator norm.  
We write $\Fourier \varphi = \int_\Real e^{-i \xi x} \varphi(x) dx$ and $\FourierInverse \varphi = \frac{1}{2\pi} \int_\Real e^{i \xi x} \varphi(x) dx$ to denote the Fourier transform and its inverse.  
%We denote by $B(X,Y)$ the space of bounded linear operators from $X$ to $Y$ equipped with the uniform operator norm.  
%We also write $B(X) = B(X,X)$.  
We denote by $L^p_t L^q_x(Q_T) = L^p(0, T; L^q(\Omega))$ the space-time Lebesgue space over $Q_T = (0, T) \times \Omega$, with the norm
\begin{align*}
    \Vert
        \varphi
    \Vert_{L^p_t L^q_x(Q_T)}
    = \Vert
        \varphi
    \Vert_{L^p(0, T; L^q(\Omega))}
    := \left(
        \int_0^T
            \Vert
                \varphi(t, \cdot)
            \Vert_{L^q(\Omega)}^p
        dt
    \right)^{1/p}.
\end{align*}
We analogously define $H^{m,p}_tH^{n,q}_x(Q_T)$ for $m, n \in \Integer_{\geq 0}$ and $p,q \in (1,\infty)$.  
We denote by $\Sigma_\phi$ a sector in $\Complex$ of angle $\phi \in (0, \pi)$ as
\begin{align*}
    \Sigma_\phi
    = \{ \lambda \in \Complex
        ; \vert
            \arg \lambda
        \vert < \phi
    \}.
\end{align*}

The paper is organized as follows.  
In Section~\ref{sec_new_formulation}, we introduce a new formulation that reduces the boundary condition in system~(\ref{eq_filter_clogging}) to a time-independent one.  
The unknowns $(v, \sigma)$ are transformed into $(\tilde{v}, \sigma)$.  
We also derive some estimates for the operator $v \mapsto \tilde{v}$.
In Section~\ref{sec_well_posedness}, we establish the well-posedness of the equation satisfied by $(\tilde{v}, \sigma)$.  
We first develop the linear theory for the corresponding linearized equation, and in particular, we prove that $A_q$ admits a bounded $H^\infty$-calculus.  
We then apply this linear theory to the nonlinear problem to prove the well-posedness of the original equations.

%%%%%%%%%%%%%%%%%%%%%%%%%%%%%%%%%%%%%%%%%%%%%%%%%%%%%%%%%%
%%%%%%%%%%%%%%%%%%%%%%%%%%%%%%%%%%%%%%%%%%%%%%%%%%%%%%%%%%
%-----------------------Section1.5-----------------------%
%%%%%%%%%%%%%%%%%%%%%%%%%%%%%%%%%%%%%%%%%%%%%%%%%%%%%%%%%%
%%%%%%%%%%%%%%%%%%%%%%%%%%%%%%%%%%%%%%%%%%%%%%%%%%%%%%%%%%
\section{New Formulation} \label{sec_new_formulation}
In this section, we reduce the equation (\ref{eq_filter_clogging}) to one having time-independent boundary conditions.
By transforming Lemma 2.9 in \cite{FurukawaKitahata2024}, we obtain the following
\begin{lemma}\label{lem_extension_theorem}
    Let $m \in \Integer_{> 0}$, $\alpha \in [0, 1]$, and $u \in H^m(I)$.
    Then there exists an increasing function $\psi_\ast \in C^\infty(-1, 1)$ satisfying $\psi_\ast(\pm 1) = \pm 1$ and extension $v \in H^m(I)$ given by
    \begin{align} \label{eq_form_of_the_extension_operator_E}
        \begin{split}
            v(x)
            & = \frac{
                1 - \alpha
            }{
                1 + (1 - \alpha)^2
            }u(x)
            + \frac{
                \psi_\ast(x)
            }{
                1 + (1 - \alpha)^2
            }u(-x)\\
            & =: E(\theta, \psi_\ast) u
        \end{split}
    \end{align}
    such that
    \begin{align} \label{eq_boundary_condition_and_estimate_of_extension}
        \begin{split}
            (1 - \alpha) \gamma_+ v
            - \gamma_- v
            = \gamma_+ u,
            & \quad
            \gamma_+ \partial_x v
            - (1 - \alpha) \gamma_- \partial_x v
            = - \gamma_- \partial_x u,\\
            \Vert
                v
            \Vert_{H^{m, q}(I)}
            & \leq C \Vert
                u
            \Vert_{H^{m, q}(I)}
        \end{split}
    \end{align}
    for some $\alpha$-independent constant $C = C(m,I)>0$.
    Moreover, we define
    \begin{align*}
        \overline{E}(\theta, \psi_\ast) u
        = \frac{
            1 - \alpha
        }{
            1 + (1 - \alpha)^2
        }u(x)
        - \frac{
            \psi_\ast(x)
        }{
            1 + (1 - \alpha)^2
        }u(-x).
    \end{align*}
    Then it follows that
    \begin{align} \label{eq_boundary_condition_and_estimate_of_extension_overline_E}
        \Vert
            \overline{E}(\theta, \psi_\ast) u
        \Vert_{H^{m, q}(I)}
        \leq C \Vert
            u
        \Vert_{H^{m, q}(I)}
    \end{align}
\end{lemma}
By definition, $E(\alpha, \varphi)$ is a linear operator, and its operator norm $\Vert \cdot \Vert_{L^q(I) \rightarrow L^q(I)}$ is bounded by $2$.
If $\varphi = \psi_\ast$, we simply denote by $E(\alpha) = E(\alpha, \psi_\ast)$.
\begin{comment}
\begin{remark} \label{rmk_how_to_take_psi_ast}
    By definition, $E(\alpha, \varphi)$ is a linear operator, and its operator norm $\Vert \cdot \Vert_{L^q(I) \rightarrow L^q(I)}$ is bounded by $2$.  
    The function $\phi_\ast$ can be chosen as an odd function that is identically equal to $-1$ on $[-1, -1/2]$, and convex on $(-1/2, 0)$.  
    The uniform norm of the derivative of such a function $\psi_\ast$ is bounded by $2$.
    Since
    \begin{align*}
        \partial_x E(\alpha) u
        = \frac{
            1 - \alpha
        }{
            1 + (1 - \alpha)^2
        } \partial_x u(x)
        + \frac{
            \psi_\ast^\prime(x)
        }{
            1 + (1 - \alpha)^2
        }u(-x)
        - \frac{
            \psi_\ast(x)
        }{
            1 + (1 - \alpha)^2
        } \partial_x u(-x)
    \end{align*}
    we observe that
    \begin{align*}
        \Vert
            E(\alpha)
        \Vert_{H^{1, q}(I) \rightarrow H^{1, q}(I)}
        & = \sup_{\Vert u \Vert_{H^{1, q}(I)} \leq 1}
        \left(
            \Vert
                E(\alpha) u
            \Vert_{L^p(I)}
            + \Vert
                \partial_x E(\alpha) u
            \Vert_{L^p(I)}
        \right)\\
        & \leq 2
        + \frac{4 - \alpha}{1 + (1 - \alpha)^2}
        \leq 6.
    \end{align*}
    In this paper, we use such $\psi_\ast$ for the definition of $E(\alpha)$.
\end{remark}
\end{comment}

\begin{proposition} \label{prop_formula_for_dxm_E}
    Let $\alpha > 0$.
    Then
    \begin{align*}
        \partial_x E(\alpha) u
        = \overline{E}(\alpha) \partial_x u
        + \frac{
            \psi_\ast^\prime (x)
        }{
            1 + (1 - \alpha)^2
        }u(-x)
    \end{align*}
    and
    \begin{align*}
        \partial_x^2 E(\alpha) u
        = E(\alpha) \partial_x^2 u
        + \frac{
            1
        }{
            1 + (1 - \alpha)^2
        } \left(
            - 2 \psi_\ast^\prime (x) \partial_x u(-x)
            + \psi_\ast^{\prime\prime} (x) u(-x)
        \right).
    \end{align*}
    Moreover, in odd case,
    \begin{align*}
        \partial_x^{2m - 1} E(\alpha) u
        & = \overline{E}(\alpha) \partial_x^{2m - 1} u\\
        & + \frac{
            1
        }{
            1 + (1 - \alpha)^2
        }
        \sum_{l = 1}^{2m - 1}
        \begin{pmatrix}
            2m - 1 \\
            l \\
        \end{pmatrix}
        (-1)^{2m - 1 - l} [\partial_x^{2m - 1 - l} u] (- x) \frac{d^l}{dx^l} \psi_\ast (x)\\
        & =: \overline{E}(\alpha) \partial_x^{2m - 1} u
        + R_{2m - 1} (\alpha) u,
    \end{align*}
    and in even case,
    \begin{align*}
        \partial_x^{2m} E(\alpha) u
        & = E(\alpha) \partial_x^{2m} u\\
        & + \frac{
            1
        }{
            1 + (1 - \alpha)^2
        }
        \sum_{l = 1}^{2m}
        \begin{pmatrix}
            2m \\
            l \\
        \end{pmatrix}
        (-1)^{2m - l} [\partial_x^{2m - l} u] (- x) \frac{d^l}{dx^l} \psi_\ast (x)\\
        & =: E(\alpha) \partial_x^{2m} u
        + R_{2m} (\alpha) u.
    \end{align*}
    The operator $R_m(\alpha)$ consists of spatial derivatives of order $(m - 1)$.
\end{proposition}
\begin{proof}
    This follows directly from the chain rule.
\end{proof}

We expand the boundary condition around $\theta_0 = \theta(\sigma_0)$ such as
\begin{align}
    \begin{split}
        B_1(\theta_0; v)
        & = (\theta - \theta_0) \gamma_+ v, \\
        B_2(\theta_0; v)
        & = - (\theta - \theta_0) \gamma_- \partial_x v.
    \end{split}
\end{align}
%The operator $G$ is commutes with $\partial_t$.
We can write the correction term $v_c$ and the difference between $v$ and $v_c$ as
\begin{gather} \label{eq_definition_of_tilde_v_vc}
    \begin{split}
        v_c
        := (\theta - \theta_0) E(\theta_0) v,\\
        \tilde{v}
        := v - v_c
        = G(\theta_0, \theta) v
        \quad \text{for} \quad
        G(\theta_0, \theta)
        := I
        - (\theta - \theta_0) E(\theta_0).
    \end{split}
\end{gather}
Clearly,
\begin{gather*}
    v
    = v_c + \tilde{v},\\
    B_j(\theta_0; \tilde{v}) = 0, \quad j = 1,2.
\end{gather*}
%We can write $v$ and $v_c$ by $\tilde{v}$ as
\begin{proposition} \label{prop_inverse_of_G}
    Assume that $\vert \theta - \theta_0 \vert \leq \delta$ for some $\delta > 0$.
    Then there exists the inverse operator $G(\theta_0, \theta)^{-1} = (I - (\theta - \theta_0) E(\theta_0))^{-1}$ such that
    \begin{align*}
        v
        & = G(\theta_0, \theta)^{-1} \tilde{v},\\
        v_c
        & = (\theta - \theta_0) E(\theta_0) G(\theta_0, \theta)^{-1} \tilde{v},
    \end{align*}
    and
    \begin{align*}
        \Vert
            G(\theta_0, \theta)^{-1}
        \Vert_{L^q(I) \rightarrow L^q(I)}
        \leq \frac{
            C
        }{
            1 - 2 (\theta - \theta_0)
        }.
    \end{align*}
    Moreover, we define
    \begin{align*}
        \overline{G}(\theta_0, \theta)
        = I - (\theta - \theta_0) \overline{E}(\theta_0).
    \end{align*}
    Then, in the same manner, there exists the inverse operator $\overline{G}(\theta_0, \theta)^{-1}$ such that
    \begin{align*}
        \Vert
            \overline{G}(\theta_0, \theta)^{-1}
        \Vert_{L^q(I) \rightarrow L^q(I)}
        \leq \frac{
            C
        }{
            1 - 2 (\theta - \theta_0)
        }.
    \end{align*}
\end{proposition}
\begin{proof}
    By the definition of $v_c$ and $\tilde{v}$, the existence of the inverse operators and their properties follow from a Neumann series argument.
\end{proof}

%Let $\sigma$ be a function of $t \geq 0$.
Applying $\partial_t - \partial_x^2$ to $v_c$ and using Proposition \ref{prop_formula_for_dxm_E}, we deduce that
\begin{align*}
    \partial_t v_c - \partial_x^2 v_c
    & = \theta^\prime(\sigma) \sigma^\prime E(\theta_0) v
    + (
        \theta(\sigma)
        - \theta_0
    ) E(\theta_0) \partial_t v
    - (
        \theta(\sigma)
        - \theta_0
    ) (
        E(\theta_0) \partial_x^2 v
        + R_2(\theta_0) v
    )\\
    & = \frac{
        \theta^\prime(\sigma) \sigma^\prime
    }{
        \theta(\sigma)
        - \theta_0
    } v_c
    + (
        \theta(\sigma)
        - \theta_0
    ) E(\theta_0) (
        f_v(v, \sigma)
        + g_v
    )
    - (
        \theta(\sigma)
        - \theta_0
    ) R_2(\theta_0) v.
\end{align*}
Since $\tilde{v} = v - v_c$, we obtain the formula
\begin{align} \label{eq_parabolic_eq_for_tilde_v}
    \begin{split}
        \partial_t \tilde{v} - \partial_x^2 \tilde{v}
        & = f_v(v, \sigma)
        + g_v
        - (
            \partial_t v_c
            - \partial_x^2 v_c
        )\\
        & = - \frac{
            \theta^\prime(\sigma) \sigma^\prime
        }{
            \theta - \theta_0
        } v_c
        + G(\theta_0, \theta) (
            f_v(v, \sigma)
            + g_v
        )
        + (\theta - \theta_0) R_2(\theta_0) v\\
        & = - \theta^\prime(\sigma) \sigma^\prime E(\theta_0) G(\theta_0, \theta)^{-1} \tilde{v}
        + G(\theta_0, \theta) f_v(G(\theta_0, \theta)^{-1} \tilde{v}, \sigma)\\
        & + G(\theta_0, \theta) g_v
        + (\theta - \theta_0) R_2(\theta_0) G(\theta_0, \theta)^{-1} \tilde{v}.
    \end{split}
\end{align}
%By the definition of $\tilde{v}$ we see that
%\begin{align*}
%    B_1(\theta_0; \tilde{v})
%    = B_2(\theta_0; \tilde{v})
%    = 0
%\end{align*}
Since $v_c(0, x) = 0$, the initial data of $\tilde{v}$ is given by
\begin{align*}
    \tilde{v}(x, 0)
    = v_0.
\end{align*}

\begin{proposition} \label{prop_Hm_to_Hm_bound_of_G}
    Let $m \in \Integer_{\geq 0}$ and $q \in (0,\infty]$.
    Let $\alpha_0 = \alpha(0)$ and $\alpha$ is a positive continues function of its range within $(0,1)$.
    Let $\mathcal{G} = \overline{G}$ if $m$ is odd and $\mathcal{G} = G$ if $m$ is even.
    Then there exists $\delta >0 $, if $t - t_0 < \delta$, the formula 
    \begin{align*}
        \partial_x^m G(\alpha_0, \alpha)^{-1} \varphi
        = \mathcal{G}(\alpha_0, \alpha)^{-1} \partial_x^m \varphi
        + (\alpha - \alpha_0) \mathcal{G}(\alpha_0, \alpha)^{-1} R_m(\alpha_0) G(\alpha_0, \alpha)^{-1} \varphi
    \end{align*}
    holds.
    Consequently,
    \begin{gather*}
        \Vert
            G(\alpha_0, \alpha)^{-1}
        \Vert_{H^{1, q}(I) \rightarrow H^{1, q}(I)}
        \leq \frac{A_2}{
            (1 - 2(\alpha - \alpha_0))^2
        }.
    \end{gather*}
    Iteratively,
    \begin{align*}
        \Vert
            G(\alpha_0, \alpha)^{-1}
        \Vert_{H^{m, q}(I) \rightarrow H^{m, q}(I)}
        \leq \frac{A_m}{
            (1 - B_m (\alpha - \alpha_0))^{m+1}
        }
    \end{align*}
    for some constant $A_m, B_m >0$.
    The constant $A_m$ and $B_m$ which is the maximum of $\Vert E(\alpha_0) \Vert_{H^{m, q}(I) \rightarrow H^{m, q}(I)}$.%, especially $B_1 = 2, \, B_2 = 6$ in view of \ref{rmk_how_to_take_psi_ast}.
\end{proposition}
\begin{proof}
    Let $\mathcal{E} = \overline{E}$ if $m$ is odd and $\mathcal{E} = E$ if $m$ is even.
    We take $\delta>0$ so small that $G(\alpha_0, \alpha)^{-1}$ exists.
    From the definition, it is immediate that
    \begin{align*}
        \varphi
        = (I - (\alpha - \alpha_0) E(\alpha_0)) G(\alpha_0, \alpha)^{-1} \varphi.
    \end{align*}
    This identity and Proposition \ref{prop_formula_for_dxm_E} yield
    \begin{align} \label{eq_formula_for_dmxG}
        \begin{split}
            \partial_x^m \varphi
            & = \partial_x^m G(\alpha_0, \alpha)^{-1} \varphi \\
            & - (\alpha - \alpha_0) \mathcal{E}(\alpha_0) \partial_x^m G(\alpha_0, \alpha)^{-1} \varphi
            - (\alpha - \alpha_0) R_m(\alpha_0) G(\alpha_0, \alpha)^{-1} \varphi,
        \end{split}
    \end{align}
    and then
    \begin{align*}
        \partial_x^m G(\alpha_0, \alpha)^{-1} \varphi
        = \mathcal{G}(\alpha_0, \alpha)^{-1} \partial_x^m \varphi
        + (\alpha - \alpha_0) \mathcal{G}(\alpha_0, \alpha)^{-1} R_m(\alpha_0) G(\alpha_0, \alpha)^{-1} \varphi.
    \end{align*}
    Multiplying both sides by $G(\alpha_0, \alpha)^{-1}$, we iteratively obtain the desired formula.
    Moreover, we observe that
    \begin{align*}
        \Vert
            G(\alpha_0, \alpha)^{-1}
        \Vert_{H^{1, q}(I) \rightarrow H^{1, q}(I)}
        & \leq
        \frac{A_0}{
            1 - 2 (\alpha - \alpha_0)
        }
        + \frac{
            C \vert
                \alpha - \alpha_0
            \vert
        }{
            (1 - 2(\alpha - \alpha_0))^2
        }\\
        & \leq \frac{A_1}{
            (1 - 2(\alpha - \alpha_0))^2
        }.
    \end{align*}
    Since
    \begin{align*}
        & \Vert
            \partial_x^m G(\alpha_0, \alpha)^{-1}
        \Vert_{H^m(I) \rightarrow H^m(I)}\\
        & \leq  
        \frac{
            A_0
        }{
            1 - 2 (\alpha - \alpha_0)
        }
        + \frac{
            C \vert
                \alpha - \alpha_0
            \vert
        }{
            1 - 2 (\alpha - \alpha_0)
        }
        \Vert
            G(\alpha_0, \alpha)^{-1}
        \Vert_{H^{m - 1, q}(I) \rightarrow H^{m - 1, q}(I)},
    \end{align*}
    we obtain the bounds on the higher order Sobolev spaces by induction.
\end{proof}
%By Proposition \ref{prop_Hm_to_Hm_bound_of_G}, the last term on the right-hand side can be estimated small as long as $\sigma - \sigma_0$ remains small.

\begin{proposition} \label{prop_estimate_to_difference_between_inverse_of_Gs}
    Let $p, q \in (1,\infty)$.
    %Let $\alpha$ is a positive function in $H^{1, p}(0, T) \hookrightarrow C[0, T]$ associated with initial trace $\alpha_0$.
    Let $\alpha_1, \alpha_2$ be positive functions in $H^{1, p}(0, T)$ associated with the initial trace $\alpha_0$.
    Assume that $\sup_{0 < t < T} \vert \alpha_j(t) - \alpha_0 \vert \leq \delta$ for some $\delta \in (0, 1)$.
    Then
    \begin{align*}
        \Vert
            G(\alpha_0, \alpha_1)^{-1}
            - G(\alpha_0, \alpha_2)^{-1}
        \Vert_{L^q(I) \rightarrow L^q(I)}
        \leq \frac{
            C \vert
                \alpha_1
                - \alpha_2
            \vert
        }{
            (1 - 2 (\alpha_1 - \alpha_0))
            (1 - 2 (\alpha_2 - \alpha_0))
        },
    \end{align*}
    and
    \begin{align*}
        & \Vert
            G(\alpha_0, \alpha_1)^{-1}
            - G(\alpha_0, \alpha_2)^{-1}
        \Vert_{H^{m, p}(I) \rightarrow H^{m, p}(I)}\\
        & \leq \frac{
            A_m^2 \vert
                \alpha_1
                - \alpha_2
            \vert
        }{
            (1 - B_m (\alpha_1 - \alpha_0))^{m+1}
            (1 - B_m (\alpha_2 - \alpha_0))^{m+1}
        }.
    \end{align*}
    Consequently,
    \begin{align} \label{eq_estimate_for_difference_of_G_inverses}
        \begin{split}
            & \Vert
                G(\alpha_0, \alpha_1)^{-1}
                - G(\alpha_0, \alpha_2)^{-1}
            \Vert_{
                L^p(0, T; H^{m, q}(I))
                \rightarrow L^p(0, T; H^{m, q}(I))} \\
            & \leq C T^{1 - 1/p} \Vert
                \alpha_1
                - \alpha_2
            \Vert_{H^{1, p}(0, T)}
        \end{split}
    \end{align}
    for some constants $A_m, B_m > 0$.
\end{proposition}
\begin{proof}
    The estimates are derived from the resolvent equation
    \begin{align*}
        G(\alpha_0, \alpha_1)^{-1}
        - G(\alpha_0, \alpha_2)^{-1}
        & = - G(\alpha_0, \alpha_1)^{-1}
        \left(
            G(\alpha_0, \alpha_1)
            - G(\alpha_0, \alpha_2)
        \right)
        G(\alpha_0, \alpha_2)^{-1}\\
        & = (
            \alpha_1
            - \alpha_2
        ) G(\alpha_0, \alpha_1)^{-1}
        G(\alpha_0, \alpha_2)^{-1}.
    \end{align*}
    and Proposition \ref{prop_Hm_to_Hm_bound_of_G}.
    The order of $T$ in (\ref{eq_estimate_for_difference_of_G_inverses}) is follows from the Sobolev embedding.
\end{proof}

\begin{proposition} \label{prop_bound_of_dtm_G}
    Let $p, q \in (1,\infty)$.
    %Let $\alpha$ is a positive function in $H^{1, p}(0, T) \hookrightarrow C[0, T]$ associated with initial trace $\alpha_0$.
    Let $\alpha$ is a positive function in $H^{1, p}(0, T)$ with the initial trace $\alpha_0$.
    Assume that $\sup_{0 < t < T} \vert \alpha(t) - \alpha_0 \vert \leq \delta$ for some $\delta \in (0, 1)$.
    Then
    \begin{align*}
        & \partial_t [G(\alpha_0, \alpha(t))^{-1} \varphi (t)]\\
        & = G(\alpha_0, \alpha(t))^{-1} \partial_t \varphi (t)
        + \frac{
            - \alpha^\prime(t)
        }{
            (1 + (1 - \alpha(t))^2)
        }
        \left(
            G(\alpha_0, \alpha(t))^{-1}
        \right)^2 \varphi\\
        & + \frac{
            2 (
                1 - \alpha(t)
            )
            \alpha^\prime(t)
        }{
            (1 + (1 - \alpha(t))^2)^2
        }
        G(\alpha_0, \alpha(t))^{-1} E(\alpha(t)) G(\alpha_0, \alpha(t))^{-1} \varphi (t)
        \quad \text{a.a. $t > 0$.}
    \end{align*}
    for all $\varphi \in H^{1, q}(0,T; L^q(I))$.
    Consequently, it follows that
    \begin{align*}
        \Vert
            \partial_t G(\alpha_0, \alpha(t))^{-1} \varphi(t)
        \Vert_{L^{q}(I)}
        \leq C
        \Vert
            \partial_t \varphi(t)
        \Vert_{L^{q}(I)}
        +
        C \vert
            \alpha^\prime(t)
        \vert
        \Vert
            \varphi(t)
        \Vert_{L^{q}(I)}, \quad
        a.a. \, t > 0,
    \end{align*}
    and
    \begin{align*}
        \Vert
            \partial_t G(\alpha_0, \alpha)^{-1} \varphi
        \Vert_{L^p(0, T; L^q(I))}
        & \leq C
        (
            1
            + \Vert
                \alpha
            \Vert_{H^{1, p}(0, T)}
        )
        \Vert
            \varphi
        \Vert_{H^{1, p}(0, T; L^q(I))}
    \end{align*}
    for some $C > 0$.
\end{proposition}
\begin{proof}
    Similar to (\ref{eq_formula_for_dmxG}), we deduce that
    \begin{align*}
        \partial_t \varphi
        & = \partial_t [G(\alpha_0, \alpha(t))^{-1} \varphi]
        - E(\alpha_0, \alpha(t)) \partial_t [G(\alpha_0, \alpha(t))^{-1} \varphi]\\
        & - \frac{
            - \alpha^\prime(t) (
                1
                + (1 - \alpha(t))^2
            )
            + 2 (
                1 - \alpha(t)
            )^2
            \alpha^\prime(t)
        }{
            (1 + (1 - \alpha(t))^2)^2
        }
        [G(\alpha_0, \alpha(t))^{-1} \varphi](x, t)\\
        & - \frac{
            2 \psi_\ast (x) (
                1 - \alpha(t)
            )
            \alpha^\prime(t)
        }{
            (1 + (1 - \alpha(t))^2)^2
        }
        [G(\alpha_0, \alpha(t))^{-1} \varphi](- x, t)\\
        & = G(\alpha_0, \alpha(t)) \partial_t [G(\alpha_0, \alpha(t))^{-1} \varphi]\\
        & - \frac{
            - \alpha^\prime(t)
        }{
            (1 + (1 - \alpha(t))^2)
        }
        [G(\alpha_0, \alpha(t))^{-1} \varphi]\\
        & - \frac{
            2 (
                1 - \alpha(t)
            )
            \alpha^\prime(t)
        }{
            1 + (1 - \alpha(t))^2
        }
        E(\alpha(t)) G(\alpha_0, \alpha(t))^{-1} \varphi.
    \end{align*}
    Then we obtain the commutator formula
    \begin{align*}
        \partial_t [G(\alpha_0, \alpha(t))^{-1} \varphi]
        & = G(\alpha_0, \alpha(t))^{-1} \partial_t \varphi
        + \frac{
            - \alpha^\prime(t)
        }{
            (1 + (1 - \alpha(t))^2)
        }
        \left(
            G(\alpha_0, \alpha(t))^{-1}
        \right)^2 \varphi\\
        & + \frac{
            2 (
                1 - \alpha(t)
            )
            \alpha^\prime(t)
        }{
            (1 + (1 - \alpha(t))^2)^2
        }
        G(\alpha_0, \alpha(t))^{-1} E(\alpha(t)) G(\alpha_0, \alpha(t))^{-1} \varphi.
    \end{align*}
    Applying $\Vert \cdot \Vert_{L^q(I)}$ and $\Vert \cdot \Vert_{L^p(0, T; L^q(I))}$ to the above equation, we also deduce the bounds.
\end{proof}

\begin{proposition} \label{prop_bound_of_difference_of_dtm_G}
    Assume the same assumptions as in Proposition \ref{prop_estimate_to_difference_between_inverse_of_Gs} for $\alpha_1, \alpha_2$, and $\alpha$.
    Then
    \begin{align*}
        & \Vert
            \partial_t \left[
                \left(
                    G(\alpha_0, \alpha_1(t))^{-1}
                    - G(\alpha_0, \alpha_2(t))^{-1}
                \right) \varphi(t)
            \right]
        \Vert_{L^q(I)}\\
        & \leq C
        \vert
            \alpha_1(t)
            - \alpha_2(t)
        \vert
        \Vert
            \partial_t \varphi(t)
        \Vert_{L^q(I)}\\
        & + C
        \left(
            (
                1
                + \sum_{j=1,2}
                    \vert
                        \alpha_j(t)
                    \vert
            )
            \vert
                \alpha_1^\prime(t)
                - \alpha_2^\prime(t)
            \vert
            +
            \sum_{j = 1,2}
                \vert
                    \alpha_j^\prime(t)
                \vert 
            \vert
                \alpha_1(t)
                - \alpha_2(t)
            \vert
        \right)
        \Vert
            \varphi(t)
        \Vert_{L^q(I)}\\
        & \quad\quad\quad\quad\quad\quad\quad\quad\quad\quad\quad\quad\quad\quad\quad\quad\quad\quad\quad\quad\quad\quad\quad\quad
        \text{for a. a. $t > 0$}
    \end{align*}
    and
    \begin{align*}
        & \Vert
            \partial_t [
                G(\alpha_0, \alpha_1)^{-1}
                - G(\alpha_0, \alpha_2)^{-1}
            ] \varphi
        \Vert_{L^p(0, T; L^q(I))}\\
        & \leq \tilde{C}
        (
            1
            + \sum_{j=1,2}
                \Vert
                    \alpha_j
                \Vert_{H^{1, p}(0,T)}
        )
        \Vert
            \alpha_1
            - \alpha_2
        \Vert_{H^{1,p}(0,T)}
        \Vert
            \varphi
        \Vert_{H^{1, p}(0, T; L^q(I))},
    \end{align*}
    for all $\varphi \in H^{1, q}(0,T; L^q(I))$.
    %some constant $C, \tilde{C} > 0$.
    The constant $\tilde{C} > 0$ tends to zero as $T \rightarrow 0$.
\end{proposition}
\begin{proof}
    We observe that
    \begin{align*}
        & \partial_t \left[
            \left(
                G(\alpha_0, \alpha_1(t))^{-1}
                - G(\alpha_0, \alpha_2(t))^{-1}
            \right) \varphi
        \right]\\
        & = [
            G(\alpha_0, \alpha_1(t))^{-1}
            - G(\alpha_0, \alpha_2(t))^{-1}
        ] \partial_t \varphi\\
        & + \left[
            \frac{
                - \alpha_1^\prime(t)
            }{
                (1 + (1 - \alpha_1(t))^2)
            }
            \left(
                G(\alpha_0, \alpha_1(t))^{-1}
            \right)^2
            - \frac{
                - \alpha_2^\prime(t)
            }{
                (1 + (1 - \alpha_2(t))^2)
            }
            \left(
                G(\alpha_0, \alpha_2(t))^{-1}
            \right)^2
        \right] \varphi\\
        & + \Biggl [
            \frac{
                2 (
                    1 - \alpha_1(t)
                )
                \alpha_1^\prime(t)
            }{
                (1 + (1 - \alpha_1(t))^2)^2
            }
            G(\alpha_0, \alpha_1(t))^{-1} E(\alpha_1(t)) G(\alpha_0, \alpha_1(t))^{-1}\\
        & \quad\quad\quad
            - \frac{
                2 (
                    1 - \alpha_2(t)
                )
                \alpha_2^\prime(t)
            }{
                (1 + (1 - \alpha_2(t))^2)^2
            }
            G(\alpha_0, \alpha_2(t))^{-1} E(\alpha_2(t)) G(\alpha_0, \alpha_2(t))^{-1}
        \Biggr ]
        \varphi.
    \end{align*}
    Let $M > 0$ be the maximum of the Lipschitz norms of the functions $\frac{1}{1 + (1 - x)^2}$ and $\frac{2(1 - x)^2}{(1 + (1 - x)^2)^2}$.  
    Then, by Proposition~\ref{prop_bound_of_dtm_G}, we find that the $L^q(I)$-norm of the third term is bounded by $\Vert \varphi \Vert_{L^q(I)}$ times
    \begin{align*}
        & C M
        \vert
            \alpha_1(t)
            - \alpha_2(t)
        \vert
        \sum_{j=1,2} 
            \vert
                \alpha_j^\prime(t)
            \vert
        + C M \vert
            \alpha_1^\prime(t)
            - \alpha_2^\prime(t)
        \vert
        \sum_{j=1,2} 
            \vert
                \alpha_j(t)
            \vert.
    \end{align*}
    Therefore,
    \begin{align*}
        & \Vert
            \partial_t [
                G(\alpha_0, \alpha_1(t))^{-1}
                - G(\alpha_0, \alpha_2(t))^{-1}
            ] \varphi
        \Vert_{L^q(I)}\\
        & \leq C
        \vert
            \alpha_1(t)
            - \alpha_2(t)
        \vert
        \Vert
            \partial_t \varphi
        \Vert_{L^q(I)}\\
        & + C M
        \left(
            (
                1
                + \sum_{j=1,2}
                    \vert
                        \alpha_j(t)
                    \vert
            )
            \vert
                \alpha_1^\prime(t)
                - \alpha_2^\prime(t)
            \vert
            +
            \sum_{j = 1,2}
                \vert
                    \alpha_j^\prime(t)
                \vert 
            \vert
                \alpha_1(t)
                - \alpha_2(t)
            \vert
        \right)
        \Vert
            \varphi
        \Vert_{L^q(I)}.
    \end{align*}
    This implies the first estimate.
    Moreover, we observe that
    \begin{align*}
        & \Vert
            \partial_t [
                G(\alpha_0, \alpha_1)^{-1}
                - G(\alpha_0, \alpha_2)^{-1}
            ] \varphi
        \Vert_{L^p(0, T; L^q(I))}\\
        & \leq C T^{1 - 1/p}
        \Vert
            \alpha_1
            - \alpha_2
        \Vert_{H^{1, p}(0,T)}
        \Vert
            \partial_t \varphi
            \Vert_{L^p(0, T; L^q(I))}\\
        & + C T^{1 - 1/p}
        (
            1
            + T^{1 - 1/p} \sum_{j=1,2}
                \Vert
                    \alpha_j
                \Vert_{H^{1, p}(0,T)}
        )
        \Vert
            \alpha_1^\prime
            - \alpha_2^\prime
        \Vert_{L^p(0,T)}
        \Vert
            \varphi
        \Vert_{H^{1, p}(0, T; L^q(I))}\\
        & + C T^{2 (1 - 1/p)}
        \sum_{j = 1,2}
            \Vert
                \alpha_j^\prime
            \Vert_{L^p(0,T)}
        \Vert
            \alpha_1
            - \alpha_2
        \Vert_{H^{1, p}(0,T)}
        \Vert
            \varphi
        \Vert_{H^{1, p}(0, T; L^q(I))}.
    \end{align*}
    This implies the second estimate.
\end{proof}
Propositions \ref{prop_Hm_to_Hm_bound_of_G}-\ref{prop_bound_of_difference_of_dtm_G} implies
\begin{corollary} \label{cor_bound_of_G_in_E_w1pqT}
    Under the same assumption of $\alpha, \alpha_1, \alpha_2$ as Proposition \ref{prop_bound_of_difference_of_dtm_G},
    \begin{align*}
        & \Vert
            G(\alpha_0, \alpha)
        \Vert_{
            H^{1, q}(0, T; L^q(I))
            \cap L^p(0, T; H^{2, q}(I))
            \rightarrow H^{1, q}(0, T; L^q(I))
            \cap L^p(0, T; H^{2, q}(I))
        }\\
        & \leq C (
            1
            + \Vert
                \alpha
            \Vert_{H^{1, p}(0, T)}
        )
    \end{align*}
    and
    \begin{align*}
        & \Vert
            G(\alpha_0, \alpha_1)
            - G(\alpha_0, \alpha_2)
        \Vert_{
            H^{1, q}(0, T; L^q(I))
            \cap L^p(0, T; H^{2, q}(I))
            \rightarrow H^{1, q}(0, T; L^q(I))
            \cap L^p(0, T; H^{2, q}(I))
        }\\
        & \leq \tilde{C}
        (
            1
            + \sum_{j=1,2}
                \Vert
                    \alpha_j
                \Vert_{H^{1, p}(0,T)}
        )
        \Vert
            \alpha_1
            - \alpha_2
        \Vert_{H^{1,p}(0,T)},
    \end{align*}
    for all $\varphi \in H^{1, q}(0,T; L^q(I))$.
    The constant $\tilde{C}$ satisfies $\tilde{C} = O(T^\eta)$ for some $\eta \in (0,1)$.
\end{corollary}

\section{Well-posedness} \label{sec_well_posedness}
Let
\begin{align} \label{eq_F_w_G_sigma}
    \begin{split}
        F_w(w, \sigma)
        & := - \theta^\prime(\sigma) \sigma^\prime E(\theta_0) G^{-1}(\theta_0, \theta) w
        + G(\theta_0, \theta) f_v(G(\theta_0, \theta)^{-1} w, \sigma)\\
        & + (\theta - \theta_0) R_2(\theta_0) G(\theta_0, \theta)^{-1} w,\\
        G_w (\sigma)
        & := G(\theta_0, \theta) g_v,\\
        F_\sigma (w, \sigma)
        & := f_\sigma(G(\theta_0, \theta)^{-1} w, \sigma).
    \end{split}
\end{align}
We consider the equations
\begin{equation} \label{eq_W}
    \begin{aligned}
        \partial_t w - \partial_x^2 w
        & = F_w(w, \sigma)
        + G_w (\sigma),
        && x \in I,
        & t > 0,\\
        B_j(\theta_0; w)
        & =0,
        && x \in \partial I,
        & \, t >0,\\
        \frac{d\sigma}{dt}
        & = F_\sigma (w, \sigma) + g_\sigma,
        && 
        & \, t >0,\\
        w(0)
        & = w_0,
        &&
        & x \in I,\\
        \sigma (0)
        & = \sigma_0 > 0,
        && 
        &
    \end{aligned}
\end{equation}
for two unknowns $w$ and $\sigma$.

\subsection{Linearized problem}
In this section, we establish the maximal regularity of $A_q$.  
We first prove the following Poincar\'{e}-type inequality.
\begin{lemma} \label{lem_poincare}
Let $D$ be a layer domain of the form $(a, b) \times D^\prime \subset \Real^d$, where $D^\prime \subset \Real^{d-1}$ is a domain and $d \in \Integer_{\geq 1}$.  
Let $q \in [1, \infty)$ and $\theta_0 \in (0, 1]$.  
Then, for all $\varphi \in H^{1,q}(D)$ satisfying
\begin{align*}
    (1 - \theta_0) \varphi \vert_{x_1 = b}
    - \varphi \vert_{x_1 = a}
    = 0,
\end{align*}
the inequality
\begin{align*}
    \Vert \varphi \Vert_{L^q(D)}
    \leq \frac{C (1 + \theta_0)}{\theta_0}
    \Vert \nabla \varphi \Vert_{L^q(D)}
\end{align*}
holds.
\end{lemma}
\begin{remark}
    As $\theta_0$ tends to zero, the boundary condition approaches the periodic boundary condition.  
    This Poincar\'{e}-type estimate reflects this behavior.
\end{remark}

\begin{proof}
Without loss of generality, we may assume $a = -1$ and $b = 1$.  
We observe that
\begin{align*}
    \varphi(x_1, x^\prime)
    &= \varphi(+1, x^\prime)
    - \int_{x_1}^{1} \partial_1 \varphi(y, x^\prime)\, dy,\\
    \varphi(x_1, x^\prime)
    &= \varphi(-1, x^\prime)
    + \int_{-1}^{x_1} \partial_1 \varphi(y, x^\prime)\, dy.
\end{align*}
By the boundary condition, it follows that
\begin{align*}
    -\theta_0\, \varphi(x_1, x^\prime)
    = - \int_{-1}^{1} \partial_1 \varphi(y, x^\prime)\, dy
    + \theta_0 \int_{x_1}^{1} \partial_1 \varphi(y, x^\prime)\, dy.
\end{align*}
Therefore,
\begin{align*}
    \Vert \varphi \Vert_{L^q(D)}
    &\leq \frac{1 + \theta_0}{\theta_0}
    \left\Vert
        \int_{-1}^{1}
        \left| \partial_1 \varphi(y, x^\prime) \right| dy
    \right\Vert_{L^q(D)}\\
    &\leq \frac{C (1 + \theta_0)}{\theta_0}
    \Vert \nabla \varphi \Vert_{L^q(D)}.
\end{align*}
\end{proof}
We consider the linearized equation
\begin{equation}
    \begin{aligned}
        \partial_t u - \partial_x^2 u
        & = f,
        && x \in I,
        & t > 0,\\
        B_j(\theta_0; u)
        & = 0,
        && x \in \partial I,
        & \, t >0,\\
        u(0)
        & = u_0,
        &&
        & x \in I.
    \end{aligned}
\end{equation}
for $\theta_0 \in (0, 1)$, an external force $f$, and initial data $u_0$.  
In \cite{FurukawaKitahata2024} the authors have proved in Lemma 2.3 that 
generates an analytic semigroup $e^{t A_q}$ in $L^2(I)$.

We consider the resolvent problem
\begin{equation} \label{eq_resolvent_problem_inhomogeneous}
    \begin{aligned}
        \lambda u - A_q u
        & = f,\\
        B_j(\theta_0; u)
        & = 0,
    \end{aligned}
\end{equation}
for $\lambda \in \Sigma_\phi$ and $f \in L^q(I)$, where $\phi \in (\pi / 2, \pi)$.  
We prove
\begin{lemma} \label{lem_resolvent_estimate_for_Aq}
    Let $\phi \in (\pi/2, \pi)$ and $\lambda \in \Sigma_\phi$.
    Then for each $f \in L^q(I)$, there exists a unique solution $u \in D(A_q)$ to (\ref{eq_resolvent_problem_inhomogeneous}) such that the resolvent estimate
    \begin{align} \label{eq_Lq_estimate_for_u}
        \vert
            \lambda
        \vert
        \Vert
            u
        \Vert_{L^q(I)}
        + \Vert
            A_q u
        \Vert_{L^q(I)}
        \leq C
        \Vert
            f
        \Vert_{L^q(I)}
    \end{align}
    holds.
    In particular, $\lambda = 0$ belongs to the resolvent set when $\theta \in (0, 1]$, and the spectrum lies on the negative real axis.
\end{lemma}
\begin{corollary}
    The resolvent set of $A_q$ is lie in a sector $\Sigma_\phi$ for some $\phi \in (\pi/2, \pi)$ such that
    \begin{align*}
        \Vert
            (
                \lambda
                - A_q
            )^{-1}
        \Vert_{L^q(I) \rightarrow L^q(I)}
        \leq \frac{
            C
        }{
            \vert
                \lambda
            \vert
        }.
    \end{align*}
    for some constant $C>0$.
    Especially, $A_q$ generates an analytic semigroup $e^{t A_q}$ in $L^q(I)$.
\end{corollary}
\begin{proof}[Proof of Lemma \ref{lem_resolvent_estimate_for_Aq}]
    We follow the same method as in Lemma~2.3 of \cite{FurukawaKitahata2024}.  
    Let $\tilde{f}$ be the zero extension of $f$ from $I$ to $\Real$.
    \begin{equation} \label{eq_linearized_resolvent_v1}
        \begin{split}
            \begin{aligned}[t]
                \lambda u_1 - \partial_x^2 u_1
                & = \tilde{f}, \quad
                x \in \Real.
            \end{aligned}
        \end{split}
    \end{equation}
    Since
    \begin{align*}
        u_1
        = \FourierInverse\!\left( \lambda + \vert \xi \vert^2 \right)^{-1}
          \Fourier \tilde{f},
    \end{align*}
    the Mikhlin theorem (Theorem~6.2.7 in \cite{Grafakos2008}) yields
    \begin{align*}
        \vert \lambda \vert \,
        \Vert u_1 \Vert_{L^q(\Real)}
        + \Vert \partial_x^2 u_1 \Vert_{L^q(\Real)}
        \leq C \Vert \tilde{f} \Vert_{L^q(\Real)}
        = C \Vert f \Vert_{L^q(I)}.
    \end{align*}
    Let $u_2$ be the solution to
    \begin{equation} \label{eq_resolvent_v2}
        \begin{split}
            \begin{aligned}[t]
                \lambda u_2 - \partial_x^2 u_2
                & = 0,\\
                B_j(\theta_0; u_2)
                & = - B_j(\theta_0; u_1).
            \end{aligned}
        \end{split}
    \end{equation}
    To state the conclusion first, $u_2$ is given by
    We find that
    \begin{align} \label{eq_formula_of_u2}
        \begin{split}
            u_2 (x)
            & = \frac{
                1
            }{
                \lambda^{1/2} M(\lambda)
            }
            (
                e^{+ \lambda^{1/2} x} \quad
                e^{- \lambda^{1/2} x}
            )
            \begin{pmatrix}
                B_{2, -}(\lambda) && - B_{1, -}(\lambda)\\
                - B_{2, +}(\lambda) && B_{1, +}(\lambda)
            \end{pmatrix}
            \begin{pmatrix}
                - B_1(\theta_0; u_1)\\
                - B_2(\theta_0; u_1)
            \end{pmatrix}\\
            & =: I_1(u_1),
        \end{split}
    \end{align}
    where
    \begin{align*}
        B_{1, \pm}(\lambda)
        & := B_1(\theta_0; e^{\pm \lambda^{1/2}x})
        = (1- \theta_0) e^{\pm \lambda^{1/2}}
        - e^{\mp \lambda^{1/2}},\\
        B_{2, \pm} (\lambda)
        & := B_2(\theta_0; e^{\pm \lambda^{1/2}x})
        = \pm \lambda^{1/2} \left(
            e^{\pm \lambda^{1/2}}
            - (1- \theta_0) e^{\mp \lambda^{1/2}}
        \right),
    \end{align*}
    and
    \begin{align*}
        M(\lambda)
        & = \left(
            (1 - \theta_0) e^{+ \lambda^{1/2}}
            - e^{- \lambda^{1/2}}
        \right)^2\\
        & + \left(
            e^{+ \lambda^{1/2}}
            - (1 - \theta_0) e^{- \lambda^{1/2}}
        \right)^2.
    \end{align*}
    See Lemma 3 \cite{FurukawaKitahata2024} for the detail of the derivation.
    Since the kernel to $(\lambda + \xi^2)^{-1}$ is given by
    \begin{align*}
        k_\lambda (x)
        & = \frac{e^{- \lambda^{1/2} \vert x \vert}}{2 \lambda^{1/2}},
    \end{align*}
    and satisfies
    \begin{align*}
        \frac{d}{dx} k_\lambda (x)
        = - \lambda^{1/2} (\mathrm{sgn} x ) k_\lambda (x)
        = \frac{
            (- \mathrm{sgn} x) e^{- \lambda^{1/2} \vert x \vert}
        }{
            2
        },
    \end{align*}
    we find that
    \begin{align} \label{eq_formula_of_u2_explicit}
        \begin{split}
            &I_1(u_1) \\
            & = \frac{
                1
            }{
                2 \lambda^{1/2} M(\lambda)
            }\\
            & \times \left[
                - (
                    e^{- \lambda^{1/2}}
                    + (1 - \theta_0) e^{+ \lambda^{1/2}}
                )
                \left(
                    (1- \theta_0)
                    \int_\Real
                        e^{
                            - \lambda^{1/2} \vert + 1 - y \vert
                            + \lambda^{1/2}x
                        }
                        \tilde{f}(y)
                \right.
                dy
            \right.\\
            & \left.
                \quad\quad\quad\quad\quad\quad\quad\quad\quad\quad\quad\quad
                    - \int_\Real
                        e^{
                            - \lambda^{1/2}\vert - 1 -y \vert
                            + \lambda^{1/2}x
                        }
                        \tilde{f}(y)
                    dy
                \right)\\
            & \quad + (
                (1 - \theta_0) e^{- \lambda^{1/2}}
                - e^{+ \lambda^{1/2}}
            )
            \left(
                \int_\Real
                    - \mathrm{sgn} (+ 1 - y) e^{
                        - \lambda^{1/2} \vert + 1 - y \vert
                        + \lambda^{1/2}x
                    }
                    \tilde{f}(y)
                dy
            \right.\\
            & \left.
                \quad\quad\quad\quad\quad\quad\quad\quad\quad
                + (1 + \theta_0) \int_\Real
                    ( - \mathrm{sgn} (- 1 - y)) e^{
                        - \lambda^{1/2} \vert - 1 - y \vert
                        + \lambda^{1/2}x
                    }
                    \tilde{f}(y)
                dy
            \right)\\
            & \quad + (
                e^{- \lambda^{1/2}}
                - (1 - \theta_0) e^{- \lambda^{1/2}}
            )
            \left(
                (1- \theta_0)
                \int_\Real
                    e^{
                        - \lambda^{1/2} \vert + 1 - y \vert
                        - \lambda^{1/2}x
                    }
                    \tilde{f}(y)
                dy
            \right.\\
            & \left.
                \quad\quad\quad\quad\quad\quad\quad\quad\quad\quad\quad\quad
                - \int_\Real
                    e^{
                        - \lambda^{1/2} \vert - 1 - y \vert
                        - \lambda^{1/2}x
                    }
                    \tilde{f}(y)
                dy
            \right)\\
            & %\left.
                \quad + (
                (1 - \theta_0) e^{+ \lambda^{1/2}}
                - e^{- \lambda^{1/2}}
            )
            \left(
                \int_\Real
                    ( - \mathrm{sgn} (+1 - y)) e^{
                        - \lambda^{1/2} \vert + 1 - y \vert
                        - \lambda^{1/2}x
                    }
                    \tilde{f}(y)
                dy
            \right.\\
            & \left.
                \quad\quad\quad\quad\quad\quad\quad\quad\quad
                - (1- \theta_0)
                \int_\Real
                    ( - \mathrm{sgn} (-1 - y)) e^{
                        - \lambda^{1/2} \vert - 1 - y \vert
                        - \lambda^{1/2}x
                    }
                    \tilde{f}(y)
                dy
            \right)
        \end{split}
    \end{align}
    Since $x \in (-1, 1)$ and $\mathrm{spt}\, \tilde{f} = [-1, 1]$, the signs of $1 \pm x$ and $\pm 1 - y$ are consistent.  
    By the Calderón–Zygmund theorem, we obtain
    \begin{align*}
        & \left \Vert
            \int_\Real
                \lambda^{1/2} e^{\lambda^{1/2}(1 \pm x)} e^{\lambda^{1/2} \vert
                    \pm 1 - y
                \vert}
                \tilde{f}(y)
            dy
        \right \Vert_{L^q(I)}
        \leq C \Vert
            f
        \Vert_{L^q(I)}
    \end{align*}
    and
    \begin{align*}
        \left \Vert
            \partial_x^2 \int_\Real
                \lambda^{-1/2} e^{\lambda^{1/2}(1 \pm x)} e^{\lambda^{1/2} \vert
                    \pm 1 - y
                \vert}
                \tilde{f}(y)
            dy
        \right \Vert_{L^q(I)}
        \leq C \Vert
            f
        \Vert_{L^q(I)}.
    \end{align*}
    Applying these estimates together with the formula for $I_1(u_1)$, we deduce that
    \begin{align*}
        \left \Vert
            \lambda u_2
        \right \Vert_{L^q(I)}
        + \left \Vert
            \partial_x^2 u_2
        \right \Vert_{L^q(I)}
        & \leq C \sup_{\lambda \in \Sigma_\phi} \left|
            \frac{
                (1- \theta_0) e^{+ \lambda^{1/2}}
                - e^{- \lambda^{1/2}}
            }{
                M(\lambda)
            }
        \right| \Vert
            f
        \Vert_{L^q(I)}\\
        & \leq C \Vert
            f
        \Vert_{L^q(I)}.
    \end{align*}
    We conclude that
    \begin{align} \label{eq_estimate_for_v}
        \vert \lambda \vert \Vert
            u
        \Vert_{L^q(I)}
        + \Vert
            \partial_x^2 u
        \Vert_{L^q(I)}
        \leq C
        \Vert
            f
        \Vert_{L^q(I)}
    \end{align}
    for some constant $C>0$.
    The solution to the eigenvalue problem for $\lambda = 0$ of $A_q$ is trivial due to the boundary conditions.  
    Since the operator $(-I + A_q)^{-1}$ is compact, the spectrum consists only of eigenvalues.  
    Any solution to the eigenvalue problem
    \begin{align*}
        A_q \varphi = \lambda \varphi, \quad \varphi \in D(A_q)
    \end{align*}
    is necessarily smooth by a bootstrapping argument.  
    Since the domain $I$ is bounded, it suffices to consider the case $q = 2$.  
    By Lemma~\ref{lem_poincare}, $\lambda$ is a non-positive real number and is bounded away from the origin.
\end{proof}

Let $h \in H^\infty_0(\Complex \setminus \overline{\Sigma_\phi})$ for some $\phi \in (\pi / 2, \pi)$.  
We consider the $H^\infty$-calculus for the operator $A_q$, which is given by
\begin{align*}
    h(A_q)
    = \frac{1}{2 \pi i} \int_{\Gamma_\phi}
        h(\lambda) (\lambda - A_q)^{-1}
    \, d\lambda,
\end{align*}
where $\Gamma_\phi$ is a contour running from $\infty e^{- i \phi}$ to $\infty e^{+i \phi}$ passing through the origin.
To analyze the operator $h(A_q)$, we return to the explicit solution formula~\eqref{eq_formula_of_u2_explicit} for the resolvent problem associated with $-A_q$.  
Let $h \in H^\infty_0(\Complex \setminus \overline{\Sigma_\phi})$ and $\varphi \in D(A_q)$.  
We observe that
\begin{align*}
    & \left \Vert
        \frac{1}{2 \pi i} \int_{\Gamma_\phi}
            \int_\Real
                \frac{
                    1
                }{
                    2 \lambda^{1/2} M(\lambda)
                }
                h(\lambda) e^{\lambda^{1/2}(1 \pm x)} e^{- \lambda^{1/2} \vert
                    \pm 1 - y
                \vert}
                \tilde{\varphi}(y)
            dy
        d\lambda
    \right \Vert_{L^q(I)}\\
    & \leq C_{\theta_0} \Vert
        h
    \Vert_{H^\infty(\Complex \setminus \overline{\Sigma_\phi} )}
    \left \Vert
        \int_\Real
            \int_0^\infty
                \frac{
                    1
                }{
                    r^{1/2} e^{2 r^{1/2} \cos \phi / 2}
                }
                e^{r^{1/2} (\cos \phi / 2) (
                        1
                        \pm x
                        - \vert
                            \pm 1 - y
                        \vert
                    )
                }
            dr
            \vert
                \tilde{\varphi}(y)
            \vert
        dy
        \right \Vert_{L^q(I)}.
\end{align*}
Since $x \in (-1, 1)$, the right-hand side is bounded by
\begin{align*}
    C_{\theta_0, \phi} \Vert
        h
    \Vert_{H^\infty(\Complex \setminus \overline{\Sigma_\phi})}
    \left \Vert
        \int_\Real
            \frac{1}{
                - 1
                \pm x
                - \vert
                    \pm 1 - y
                \vert
            } \vert
                \tilde{\varphi}(y)
            \vert
        dy
    \right \Vert_{L^q(I)}.
\end{align*}
Since $\tilde{\varphi}$ is supported in $(-1, 1)$, and using the boundedness of the Hilbert transform, we deduce that
\begin{align*}
    & \left \Vert
        \frac{1}{2 \pi i} \int_{\Gamma_\phi}
            \int_\Real
                \frac{
                    1
                }{
                    2 \lambda^{1/2} M(\lambda)
                }
                h(\lambda) e^{\lambda^{1/2}(1 \pm x)} e^{- \lambda^{1/2} \vert
                    \pm 1 - y
                \vert}
                \vert
                    \tilde{\varphi}(y)
                \vert
            dy
        d\lambda
    \right \Vert_{L^q(I)}\\
    & \leq C_{\theta_0, \phi} \Vert
        h
    \Vert_{H^\infty(\Complex \setminus \overline{\Sigma_\phi})}
    \Vert
        \varphi
    \Vert_{L^q(I)}.
\end{align*}
Repeating the same type of estimate for all terms appearing in the expression~\eqref{eq_formula_of_u2_explicit}, we obtain
\begin{align} \label{eq_H_infty_estimate_for_Aq}
    \Vert h(A_q) \Vert_{L^q(I) \rightarrow L^q(I)}
    \leq C_{\theta_0, \phi} \Vert h \Vert_{H^\infty(\Complex \setminus \overline{\Sigma_\phi})}
\end{align}
for all $h \in H^\infty_0(\Complex \setminus \overline{\Sigma_\phi})$.
Summing up these observations, we obtain
\begin{theorem} \label{thm_H_infty_calculus_of_A}
The operator $A_q$ with domain $D(A_q)$ admits a bounded $H^\infty$-calculus with zero $H^\infty$-angle, i.e., for all $\phi \in (\pi/2, \pi)$ and $h \in H^\infty(\Complex \setminus \overline{\Sigma_\phi})$, the operator $h(A_q)$ satisfies the estimate~\eqref{eq_H_infty_estimate_for_Aq}.
\end{theorem}

\begin{corollary}
Under the same assumptions as in Theorem~\ref{thm_H_infty_calculus_of_A}, the sum of closed operators
\begin{align*}
    A_q + \Delta_{\Real^d, q}
\end{align*}
defined as the Laplace operator in the $(d+1)$-dimensional infinite layer $\Omega := (-1, 1) \times \Real^d$, admits a bounded $H^\infty$-calculus on $L^q(\Omega)$ for all $q \in (1, \infty)$, where the domain of $\Delta_{\Real^d, q}$ is $H^{2, q}(\Real^d)$.
\end{corollary}

\begin{proof}
This follows from the joint $H^\infty$-functional calculus for $A_q$ and $\Delta_{\Real^d, q}$.  
See, for example, Section~5.2 of \cite{PrussSimonett2016} and Theorem~10.5 in \cite{Nau2012}.
\end{proof}

\begin{corollary}
    Under the same assumptions of \ref{thm_H_infty_calculus_of_A}, the sum of closed operators
    \begin{align*}
        A_q + \Delta_{\Torus^d, q}
    \end{align*}
    defined as a Laplace operator in the $(d+1)$-dimensional infinite layer $\Omega = (-1, 1) \times \Torus^d$ admits admits a bounded $H^\infty$-calculus in $L^q(\Omega)$ for all $q \in (1, \infty)$, where the domain of $\Delta_{\Torus^d, q}$ is $H^{2, q}_0(\Torus^d) = \{\varphi \in H^{2, q}(\Torus^d) \, ; \, \int_{\Torus^d} \varphi(x) dx = 0\}$.
\end{corollary}
Let
\begin{gather*}
    E_{w, 0, p, q, T}
    = L^p(0,T; L^q(I)), \quad
    E_{w, 1, p, q, T}
    = H^{1, p}(0,T; L^q(I))
    \cap L^p(0,T; D(A_q)),\\
    E_{\sigma, 0, p, T}
    = H^{1, p}(0,T), \quad
    E_{\sigma, 1, p, T}
    = L^{p}(0,T).
\end{gather*}
We denote by $B_w(r)$ and $B_\sigma(r)$ the closed balls of radius $r > 0$ centered at the origin in $E_{w, 1, p, q, T}$ and $E_{\sigma, 1, p, T}$, respectively.  
Theorem~\ref{thm_H_infty_calculus_of_A} implies maximal $L^p$–$L^q$ regularity for the associated linear parabolic equation.
\begin{corollary} \label{cor_maximal_regularity_for_linearized_equation}
Let $1 < p, q < \infty$.  
Let $w_0 \in D_{p, q, 1 - 1/p}$ and $g_w \in E_{w, 0, p, q, T}$.  
Then there exist constants $C_0, C > 0$ such that the problem
\begin{equation} \label{eq_linear_w}
    \begin{aligned}
        \frac{d w}{dt} - A_q w &= g_w, \quad &&0 < t < T, \\
        w(0) &= w_0,
    \end{aligned}
\end{equation}
admits a unique solution
\begin{gather*}
    w \in E_{w, 1, p, q, T}
    \cap C((0, T); D(A_q))
    \cap C([0, T); D_{p, q, 1 - 1/p}),
\end{gather*}
such that
\begin{align*}
    \Vert w \Vert_{E_{w, 1, p, q, T}}
    \leq C_0 \Vert w_0 \Vert_{D_{p, q, 1 - 1/p}}
    + C \Vert g_w \Vert_{E_{w, 0, p, q, T}}.
\end{align*}
\end{corollary}

\begin{remark}
The constant $C_0$ tends to zero as $T \to 0$.  
Assume that $g_w \equiv 0$ and suppose, for contradiction, that this is not the case.  
Then there exists a constant $\delta > 0$ such that
\begin{align*}
    \sup_{\Vert w_{0,j} \Vert_{D_{p, q, 1 - 1/p}} = 1}
        \Vert w \Vert_{E_{w, 1, p, q, T}} \geq \delta,
\end{align*}
where $w$ denotes the solution to \eqref{eq_linear_w}.  
Note that the left-hand side is uniformly bounded for $T \in (0, 1)$, while the norm $\Vert w \Vert_{E_{w, 1, p, q, T}}$ involves integration over the interval $(0, T)$.  
Hence, taking the limit $T \to 0$ yields a contradiction, since the norm must tend to zero.  
\end{remark}

\subsection{Nonlinear problems}
We apply Corollary \ref{cor_maximal_regularity_for_linearized_equation} to obtain
\begin{theorem} \label{thm_well_posedness_nonlinear_propblem_abstract}
    Let $1 < p, q < \infty$ and $l, m \in \Integer_{\geq 1}$.
    Let $w_0 \in D_{p, q, 1 - 1/p}$, $\sigma_0 \in \Real$, $\theta_0 = \theta(\sigma_0)$ and $A_q = A_q(\theta_0)$.
    Let $g_w \in E_{w, 0, p, q, T}$ and $g_\sigma \in E_{\sigma, 0, p, T}$.
    Assume that $\overline{F_w}(w, \sigma)$ and $\overline{F_\sigma}(w, \sigma)$ satisfy
    \begin{gather} \label{eq_assumption_1_thm_well_posedness_nonlinear_propblem_abstract}
        \begin{split}
            \Vert
                \overline{F_w}(w, \sigma)
            \Vert_{E_{w, 0, p, q, T}}
            \leq C (
                \delta(T)
                + \omega^l
                + s^l
            )
            (
                \omega
                + s
            ),\\
            \Vert
                \overline{F_\sigma}(w, \sigma)
            \Vert_{L^p(0, T)}
            \leq C (
                \delta(T)
                + \omega^m
                + s^m
            )
            (
                \omega
                + s
            ),\\
            (
                \omega
                = \Vert
                    w
                \Vert_{E_{w, 1, p, q, T}}, \,
                s = \Vert
                    \sigma
                \Vert_{E_{\sigma, 1, p, T}}
            )
        \end{split}
    \end{gather}
    and
    \begin{align} \label{eq_assumption_2_thm_well_posedness_nonlinear_propblem_abstract}
        \begin{split}
            & \Vert
                \overline{F_w}(w_1, \sigma_1)
                - \overline{F_w}(w_2, \sigma_2)
            \Vert_{E_{w, 0, p, q, T}}\\
            & \quad \leq C (
                \delta(T)
                + \omega_1^l
                + \omega_2^l
                + s_1^l
                + s_2^l
            )
            (
                \Vert
                    w_1
                    - w_2
                \Vert_{E_{w, 1, p, q, T}}
                + \Vert
                    \sigma_1
                    - \sigma_2
                \Vert_{_{\sigma, 1, p, T}}
            ),\\
            & \Vert
                \overline{F_\sigma}(w_1, \sigma_1)
                - \overline{F_\sigma}(w_2, \sigma_2)
            \Vert_{E_{w, 0, p, q, T}}\\
            & \quad \leq C (
                \delta(T)
                + \omega_1^m
                + \omega_2^m
                + s_1^m
                + s_2^m
            )
            (
                \Vert
                    w_1
                    - w_2
                \Vert_{E_{w, 1, p, q, T}}
                + \Vert
                    \sigma_1
                    - \sigma_2
                \Vert_{E_{\sigma, 1, p, T}}
            ),\\
            & \quad\quad\quad\quad\quad\quad
            (
                \omega_j
                = \Vert
                    w_j
                \Vert_{E_{w, 1, p, q, T}}, \,
                s_j
                = \Vert
                    \sigma_j
                \Vert_{E_{\sigma, 1, p, T}}
            )
        \end{split}
    \end{align}
    for some constant $\delta(T) = o(1)$ converging to zero as $T \rightarrow 0$.
    Then there exists $T_0 \in (0, T)$ and a unique solution
    \begin{gather*}
        w
        \in E_{w, 1, p, q, T_0}
        \cap C((0, T_0); D(A_q))
        \cap C([0, T_0); D_{p, q, 1 - 1 / p}),\\
        \sigma
        \in H^{1, p}(0,T_0)
        \hookrightarrow C[0, T_0)
    \end{gather*}
    to
    \begin{equation} \label{eq_W_generalized}
        \begin{aligned}
            \frac{d w}{dt} - A_q w
            & = \overline{F_w}(w, \sigma)
            + g_w,
            & t > 0, \\
            \frac{d\sigma}{dt}
            & = \overline{F_\sigma} (w, \sigma) + g_\sigma, 
            & t > 0, \\
            w(0)
            & = w_0, \\
            \sigma (0)
            & = \sigma_0
        \end{aligned}
    \end{equation}
    such that
    \begin{align*}
        & \Vert
            w
        \Vert_{E_{w, 1, p, q, T_0}}
        + \Vert
            \sigma
        \Vert_{E_{\sigma, 1, p, T}}\\
        & \leq C_{T} (
            \Vert
                w_0
            \Vert_{D_{p, q, 1 - 1 / p}}
            + \vert
                \sigma_0
            \vert
        )
        + C (
            \Vert
                g_w
            \Vert_{E_{w, 0, p, q, T_0}}
            + \Vert
                g_\sigma
            \Vert_{E_{\sigma, 0, p, T}}
        ),
    \end{align*}
    where $C_T$ tends to zero as $T \rightarrow 0$.
\end{theorem}
\begin{proof}
    The proof follows the classical Fujita–Kato approach.  
    By the maximal $L^p$–$L^q$ regularity of $A_q$ Corollary~\ref{cor_maximal_regularity_for_linearized_equation} and the assumption~\eqref{eq_assumption_1_thm_well_posedness_nonlinear_propblem_abstract}, the solution mapping
    \begin{align*}
        \mathcal{S} : (\overline{w}, \overline{\sigma})
        \rightarrow (w, \sigma)
    \end{align*}
    associated with the system
    \begin{equation}
        \begin{aligned}
            \frac{d w}{dt} - A_q w
            & = \overline{F_w}(\overline{w}, \overline{\sigma})
            + g_w,
            & t > 0, \\
            \frac{d\sigma}{dt}
            & = \overline{F_\sigma} (\overline{w}, \overline{\sigma}) + g_\sigma, 
            & t > 0, \\
            w(0)
            & = w_0, \\
            \sigma (0)
            & = \sigma_0
        \end{aligned}
    \end{equation}
    is bounded in $E_{w, 0, p, q, T_0} \times E_{\sigma, 0, p, T}$ and satisfies
    \begin{align} \label{eq_iteration}
        \begin{split}
            & \Vert
                w
            \Vert_{E_{w, 1, p, q, T_0}}
            + \Vert
                \sigma
            \Vert_{E_{\sigma, 1, p, T}}\\
            & \leq C_\delta (
                \Vert
                    \overline{w}
                \Vert_{E_{w, 1, p, q, T_0}}
                + \Vert
                    \overline{\sigma}
                \Vert_{E_{\sigma, 1, p, T}}
            )
            + C (
                \Vert
                    \overline{w}
                \Vert_{E_{w, 1, p, q, T_0}}
                + \Vert
                    \overline{\sigma}
                \Vert_{E_{\sigma, 1, p, T}}
            )^{M}\\
            & + C_{T} (
                \Vert
                    w_0
                \Vert_{D_{p, q, 1 - 1 / p}}
                + \vert
                    \sigma_0
                \vert
            )
            + C (
                \Vert
                    g_w
                \Vert_{E_{w, 0, p, q, T_0}}
                + \Vert
                    g_\sigma
                \Vert_{E_{\sigma, 0, p, T}}
            ),
        \end{split}
    \end{align}
    where $M = 1 + \max(m, l)$.
    The third line of \eqref{eq_iteration}, denoted by $a_0$, tends to zero as $T \to 0$.  
    Let $y_\ast \in \Real$ be the smallest positive zero of the polynomial $f(y) = a_0 - (1 - C_\delta) y + C y^M$. Such a zero exists and is small due to the smallness of $a_0$ and $C_\delta$.  
    Then the restriction of $\mathcal{S}$ to the closed ball of radius $y_\ast$ in the product space $E_{w, 1, p, q, T_0} \times E_{\sigma, 1, p, T}$ satisfies
        \begin{align*}
        \Vert
            \mathcal{S}(\overline{w}, \overline{\sigma})
        \Vert_{
            E_{w, 1, p, q, T_0} \times 
            E_{\sigma, 1, p, T}
        }
        \leq C_\delta y_\ast
        + C y_\ast^M
        + a_0
        = y_\ast.
    \end{align*}
    Moreover, by assumption~\eqref{eq_assumption_2_thm_well_posedness_nonlinear_propblem_abstract}, the mapping $\mathcal{S}$ is Lipschitz continuous in this ball.  
    Therefore, by the Banach fixed point theorem, $\mathcal{S}$ admits a unique fixed point in this ball.
\end{proof}

\begin{lemma} \label{lem_nonlinear_problem}
    Let $0 < \delta < 1/p$.
    Assume the same assumptions for $w_0$, $g_w$, and $g_\sigma$ as Theorem \ref{thm_well_posedness_nonlinear_propblem_abstract}.
    Assume that $f_v$ and $f_\sigma$ satisfy \eqref{eq_assumption_1_thm_well_posedness_nonlinear_propblem_abstract} and \eqref{eq_assumption_2_thm_well_posedness_nonlinear_propblem_abstract}.
    Then there exists $T_0 \in (0, T)$ and a unique solution
    \begin{gather*}
        w
        \in E_{w, 1, p, q, T_0}
        \cap C((0, T_0); D(A_q))
        \cap C([0, T_0); D_{p, q, 1 - 1 / p}),\\
        \sigma
        \in E_{\sigma, 1, p, T}
    \end{gather*}
    to (\ref{eq_W}) associated with the nonlinear term \eqref{eq_F_w_G_sigma} such that
    \begin{align*}
        & \Vert
            w
        \Vert_{E_{w, 1, p, q, T_0}}
        + \Vert
            \sigma
        \Vert_{E_{\sigma, 1, p, T_0}}\\
        & \leq C_{T_0} (
            \Vert
                w_0
            \Vert_{D_{p, q, 1 - 1 / p}}
            + \vert
                \sigma_0
            \vert
            + \Vert
                g_w
            \Vert_{E_{w, 0, p, q, T_0}}
            + \Vert
                g_\sigma
            \Vert_{E_{\sigma, 0, p, T_0}}
        ).
    \end{align*}
\end{lemma}
\begin{proof}
%    The system is
%    \begin{equation}
%        \begin{aligned}
%            \partial_t w - \partial_x^2 w
%            & = F_w(w, \sigma)
%            + G_w (\sigma),
%            && x \in I,
%            & t > 0,\\
%            B_j(\theta_0; w)
%            & =0,
%            && x \in \partial I,
%            & \, t >0,\\
%            \frac{d\sigma}{dt}
%            & = F_\sigma (w, \sigma) + g_\sigma,
%            && 
%            & \, t >0,\\
%            w(0)
%            & = w_0,
%            &&
%            & x \in I,\\
%            \sigma (0)
%            & = \sigma_0 > 0.
%            && 
%            &
%        \end{aligned}
%    \end{equation}
%    We show
%    \begin{align*}
%        \Vert
%            F_w(w, \sigma)
%        \Vert_{E_{w, 0, p, q, T}}
%        & \leq C T^\eta \Vert
%            w
%        \Vert_{E_{w, 1, p, q, T}}
%        \Vert
%            \sigma
%        \Vert_{E_{\sigma, 1, p, T}}, \\
%        \Vert
%            F_w(w_1, \sigma_1)
%            - F_w(w_2, \sigma_2)
%        \Vert_{E_{w, 0, p, q, T}}
%        & \leq C T^\eta \sum_{j=1,2} \left(
%            \Vert
%                w_j
%            \Vert_{E_{w, 1, p, q, T}}
%            + \Vert
%                \sigma_j
%            \Vert_{E_{\sigma, 1, p, T}}
%        \right)\\
%        & \times \left(
%            \Vert
%                  w_1
%                - w_2
%            \Vert_{E_{w, 1, p, q, T}}
%            + \Vert
%                  \sigma_1
%                - \sigma_2
%            \Vert_{E_{\sigma, 1, p, T}}
%        \right)
%    \end{align*}
    Let $w, w_1, w_2 \in B_w(r)$ and $\sigma, \sigma_1, \sigma_2$.
    We denote the forcing terms by
    \begin{align}
        \begin{split}
            F_w^1(w, \sigma)
            & = - \theta^\prime(\sigma) \sigma^\prime E(\theta_0) G^{-1}(\theta_0, \theta(\sigma)) w,\\
            F_w^2(w, \sigma)
            & = G(\theta_0, \theta(\sigma)) f_v(G(\theta_0, \theta(\sigma))^{-1} w, \sigma),\\
            F_w^3(w, \sigma)
            & = (\theta(\sigma) - \theta_0) R_2(\theta_0) G(\theta_0, \theta(\sigma))^{-1} w,\\
            G_w (\sigma)
            & = G(\theta_0, \theta(\sigma)) g_v,\\
            F_\sigma (w, \sigma)
            & = f_\sigma(G(\theta_0, \theta(\sigma))^{-1} w, \sigma).
        \end{split}
    \end{align}
    The term $G_w(\sigma)$ can be easily estimated by Corollary~\ref{cor_bound_of_G_in_E_w1pqT} as
    \begin{align*}
        & \Vert
            G_w (\sigma)
        \Vert_{E_{w, 0, p, q, T}}
        \leq C
        \Vert
            g_v
        \Vert_{E_{w, 0, p, q, T}}
        \left(
            1
            + \Vert
                \sigma
            \Vert_{E_{\sigma, 1, p, T}}
        \right), \\
        & \Vert
              G_w (\sigma_1)
            - G_w (\sigma_2)
        \Vert_{E_{w, 0, p, q, T}}\\
        & \leq C
        \Vert
            g_v
        \Vert_{E_{w, 0, p, q, T}}
        \left(
            1
            + \sum_{j = 1,2} \Vert
                \sigma
            \Vert_{E_{\sigma, 1, p, T}}
        \right)
        \Vert
              \sigma_1
            - \sigma_2
        \Vert_{E_{\sigma, 1, p, T}}.
    \end{align*}
    We begin by observing that
    \begin{align*}
        \vert
            \theta(\sigma(t))
            - \theta_0
        \vert
        = \left \vert
            \int_0^t
                \theta^\prime(\sigma(\tau)) \sigma^\prime(\tau)
            d\tau
        \right \vert
        \leq C T^{1 - 1/p} \Vert
            \sigma^\prime
        \Vert_{L^p(0,T)}.
    \end{align*}
    Let $r > 0$ be chosen such that $2 C T^{1 - 1/p} r \leq 1/2$, so that the inverse operator $G(\theta_0, \theta(\sigma))^{-1}$ exists by Proposition~\ref{prop_inverse_of_G}.
    We then estimate
    \begin{align*}
        \Vert
            F_w^1(w, \sigma)
        \Vert_{E_{w, 0, p, q, T}}
        & \leq C
        \Vert
            \sigma^\prime
        \Vert_{L^p(0,T)}
        \Vert
            w
        \Vert_{L^\infty(0, T; L^q(I))}\\
        & \leq C
        \Vert
            \sigma
        \Vert_{E_{\sigma, 1, p, T}}
        \Vert
            w
        \Vert_{E_{w, 1, p, q, T}}.
    \end{align*}
    The assumption on $f_v$ implies
    \begin{align*}
        & \Vert
            F_w^2(w, \sigma)
        \Vert_{E_{w, 0, p, q, T}}\\
        & \leq C
        \left(
            \delta(T)
            + \Vert
                G(\theta_0, \theta)^{-1} w
            \Vert_{E_{w, 1, p, q, T}}^l
            + \Vert
                \sigma
            \Vert_{_{\sigma, 1, p, T}}^l
        \right) \\
        & \quad\quad\quad\quad\quad \times \left(
            \Vert
                G(\theta_0, \theta)^{-1} w
            \Vert_{E_{w, 1, p, q, T}}
            + \Vert
                \sigma
            \Vert_{_{\sigma, 1, p, T}}
        \right)\\
        & \leq C
        \left(
            \delta(T)
            +
            (
                1
                + \Vert
                    \sigma
                \Vert_{E_{\sigma, 1, p, T}}
            )^l
            \Vert
                w
            \Vert_{E_{w, 1, p, q, T}}^l
            + \Vert
                \sigma
            \Vert_{_{\sigma, 1, p, T}}^l
        \right)\\
        & \quad\quad\quad\quad\quad\quad \times
        \left(
            (
                1
                + \Vert
                    \sigma
                \Vert_{E_{\sigma, 1, p, T}}
            )
            \Vert
                w
            \Vert_{E_{w, 1, p, q, T}}
            + \Vert
                \sigma
            \Vert_{_{\sigma, 1, p, T}}
        \right)\\
        & \leq C
        \left(
            \delta(T)
            +
            \Vert
                w
            \Vert_{E_{w, 1, p, q, T}}^{2l}
            + \Vert
                \sigma
            \Vert_{_{\sigma, 1, p, T}}^{2l}
        \right)
        \left(
            \Vert
                w
            \Vert_{E_{w, 1, p, q, T}}^2
            + \Vert
                \sigma
            \Vert_{_{\sigma, 1, p, T}}^2
        \right).
    \end{align*}
    In the same way, we similarly obtain the estimate
    \begin{align*}
        & \Vert
            F_\sigma(w, \sigma)
        \Vert_{E_{w, 0, p, q, T}}\\
        & \leq C
        \left(
            \delta(T)
            + \Vert
                w
            \Vert_{E_{w, 0, p, q, T}}^{2m}
            + \Vert
                \sigma
            \Vert_{_{\sigma, 1, p, T}}^{2m}
        \right)
        \left(
            \Vert
                w
            \Vert_{E_{w, 0, p, q, T}}^2
            + \Vert
                \sigma
            \Vert_{_{\sigma, 1, p, T}}^2
        \right).
    \end{align*}
    In addition, by Proposition \ref{prop_Hm_to_Hm_bound_of_G}, we obtain
    \begin{align*}
        \Vert
            F_w^3(w, \sigma)
        \Vert_{E_{w, 0, p, q, T}}
        %& \leq C T^{1 - 1/p}
        %\Vert
        %    \sigma
        %\Vert_{H^{1, p}}
        %(
        %    1
        %    + \Vert
        %        \sigma
        %    \Vert_{H^{1, p}}
        %)
        %\Vert
        %    w
        %\Vert_{L^p(0, T; H^1(I))}\\
        \leq C T^{1 - 1/p}
        \Vert
            \sigma
        \Vert_{E_{\sigma, 1, p, T}}
        (
            1
            + \Vert
                \sigma
            \Vert_{E_{\sigma, 1, p, T}}
        )
        \Vert
            w
        \Vert_{E_{w, 1, p, q, T}}.
    \end{align*}
    We next show the Lipschitz continuity of \( F_w^1(\cdot, \cdot) \).
    We decompose the difference as
    \begin{align*}
        & F_w^1(w_1, \sigma_1)
        - F_w^1(w_2, \sigma_2)\\
        & = - \left[
            \theta^\prime(\sigma_1) \sigma_1^\prime E(\theta_0) G^{-1}(\theta_0, \theta(\sigma_1)) w_1
            - \theta^\prime(\sigma_2) \sigma_2^\prime E(\theta_0) G^{-1}(\theta_0, \theta(\sigma_2)) w_2
        \right]\\
        & = - \left[
            \theta^\prime(\sigma_1)
            - \theta^\prime(\sigma_2)
        \right]
        \sigma_1^\prime E(\theta_0) G^{-1}(\theta_0, \theta(\sigma_1)) w_1 \\
        & - \theta^\prime(\sigma_2)
        (
            \sigma_1^\prime
            - \sigma_2^\prime
        )
        E(\theta_0) G^{-1}(\theta_0, \theta(\sigma_1)) w_1\\
        & - \theta^\prime(\sigma_2) \sigma_2^\prime E(\theta_0)
        \left[
            G^{-1}(\theta_0, \theta(\sigma_1))
            - G^{-1}(\theta_0, \theta(\sigma_2))
        \right]
        w_1\\
        & - \theta^\prime(\sigma_2) \sigma_2^\prime E(\theta_0) G^{-1}(\theta_0, \theta(\sigma_2))
        (
            w_1
            - w_2
        ).
    \end{align*}
    By Propositions~\ref{prop_bound_of_dtm_G} and~\ref{prop_bound_of_difference_of_dtm_G}, we obtain
    \begin{align*}
        & \Vert
            F_w^1(w_1, \sigma_1)
            - F_w^1(w_2, \sigma_2)
        \Vert_{E_{w, 0, p, q, T}}\\
        & \leq C T^{1 - 1/p}
        \Vert
            \sigma_1
            - \sigma_2
        \Vert_{H^{1, p}(0,T)}
        \Vert
            \sigma_1
        \Vert_{E_{\sigma, 1, p, T}}
        \Vert
            G(\theta_0, \theta(\sigma_1))^{-1} w_1
        \Vert_{H^{1, p}(0, T; L^q(I))}\\
        & + C
        \Vert
            \sigma_1
            - \sigma_2
        \Vert_{E_{\sigma, 1, p, T}}
        \Vert
            G(\theta_0, \theta(\sigma_1))^{-1} w_1
        \Vert_{H^{1, p}(0, T; L^q(I))}\\
        & + C T^{1 - 1/p}
        \Vert
            \sigma_2
        \Vert_{E_{\sigma, 1, p, T}}
        \Vert
            (
                G^{-1}(\theta_0, \theta(\sigma_1))
                - G^{-1}(\theta_0, \theta(\sigma_2))
            )
            w_1
        \Vert_{H^{1, p}(0, T; L^q(I))}\\
        & + C
        \Vert
            \sigma_2
        \Vert_{E_{\sigma, 1, p, T}}
        \Vert
            G^{-1}(\theta_0, \theta(\sigma_2))
            (
                w_1
                - w_2
            )
        \Vert_{H^{1,p}(0, T; L^q(I))}\\
        & \leq C T^{1 - 1/p}
        \Vert
            \sigma_1
            - \sigma_2
        \Vert_{H^{1, p}(0,T)}
        \Vert
            \sigma_1
        \Vert_{E_{\sigma, 1, p, T}}
        (
            1
            + \Vert
                \sigma_1
            \Vert_{H^{1, p}(0, T)}
        )
        \Vert
            w_1
        \Vert_{H^{1, p}(0, T; L^q(I))}\\
        & + C
        \Vert
            \sigma_1
            - \sigma_2
        \Vert_{E_{\sigma, 1, p, T}}
        (
            1
            + \Vert
                \sigma_1
            \Vert_{H^{1, p}(0, T)}
        )
        \Vert
            w_1
        \Vert_{H^{1, p}(0, T; L^q(I))}\\
        & + C T^{1 - 1/p}
        \Vert
            \sigma_2
        \Vert_{E_{\sigma, 1, p, T}}
        \left(
            1
            + \sum_{j=1,2}
                \Vert
                    \sigma_j
                \Vert_{H^{1, p}(0,T)}
        \right)
        \Vert
            \sigma_1
            - \sigma_2
        \Vert_{H^{1,p}(0,T)}
        \Vert
            w_1
        \Vert_{H^{1, p}(0, T; L^q(I))}\\
        & + C
        \Vert
            \sigma_2
        \Vert_{E_{\sigma, 1, p, T}}
        \left(
            1
            + \Vert
                \sigma_2
            \Vert_{H^{1, p}(0, T)}
        \right)
        \Vert
            w_1
            - w_2
        \Vert_{H^{1,p}(0, T; L^q(I))}.
    \end{align*}
    Similarly, we split
    \begin{align*}
        & F_w^2(w_1, \sigma_1)
        - F_w^2(w_2, \sigma_2)\\
        %& = G(\theta_0, \theta(\sigma_1)) f_v(G(\theta_0, \theta(\sigma_1))^{-1} w_1, \sigma_1)
        %- G(\theta_0, \theta(\sigma_2)) f_v(G(\theta_0, \theta(\sigma_2))^{-1} w_2, \sigma_2)\\
        & =
        [
            G(\theta_0, \theta(\sigma_1))
            - G(\theta_0, \theta(\sigma_2))
        ]
        f_v(G(\theta_0, \theta(\sigma_1))^{-1} w_1, \sigma_1)\\
        & + G(\theta_0, \theta(\sigma_2)) 
        [
            f_v(G(\theta_0, \theta(\sigma_1))^{-1} w_1, \sigma_1)
            - f_v(G(\theta_0, \theta(\sigma_1))^{-1} w_2, \sigma_2)
        ]\\
        & + G(\theta_0, \theta(\sigma_2)) 
        [
            f_v(G(\theta_0, \theta(\sigma_1))^{-1} w_2, \sigma_1)
            - f_v(G(\theta_0, \theta(\sigma_2))^{-1} w_2, \sigma_2)
        ],
    \end{align*}
    By Corollary~\ref{cor_bound_of_G_in_E_w1pqT}, we obtain the estimate
    \begin{align*}
        & \Vert
            F_w^2(w_1, \sigma_1)
            - F_w^2(w_2, \sigma_2)
        \Vert_{E_{w, 0, p, q, T}}\\
        & \leq C T^{1 - 1/p}
        \Vert
              \sigma_1
            - \sigma_2
        \Vert_{H^{1, p}(0,T)}
        (
            \delta(T)
            + \Vert
                G(\theta_0, \theta(\sigma_1))^{-1} w_1
            \Vert_{E_{w, 1, p, q, T}}^l
            + \Vert
                \sigma_1
            \Vert_{E_{\sigma, 1, p, T}}^l
        )\\
        & \quad\quad\quad\quad \times
        (
            \Vert
                G(\theta_0, \theta(\sigma_1))^{-1} w_1
            \Vert_{E_{w, 1, p, q, T}}
            + \Vert
                \sigma_1
            \Vert_{E_{\sigma, 1, p, T}}
        )\\
        & + C
        \sum_{j = 1, 2}
        (
            \delta(T)
            + \Vert
                G(\theta_0, \theta(\sigma_1))^{-1} w_j
            \Vert_{E_{w, 1, p, q, T}}^l
            + \Vert
                \sigma_j
            \Vert_{E_{\sigma, 1, p, T}}^l
        )\\
        & \quad\quad\quad\quad \times
        \left(
            \Vert
                G(\theta_0, \theta(\sigma_1))^{-1} (
                    w_1
                    - w_2
                )
            \Vert_{E_{w, 1, p, q, T}}
            + \Vert
                \sigma_1
                - \sigma_2
            \Vert_{E_{\sigma, 1, p, T}}
        \right)\\
        & + C
        \sum_{j = 1, 2}
        (
            \delta(T)
            + \Vert
                G(\theta_0, \theta(\sigma_1))^{-1} w_j
            \Vert_{E_{w, 1, p, q, T}}^l
            + \Vert
                \sigma_j
            \Vert_{E_{\sigma, 1, p, T}}^l
        )\\
        & \quad\quad\quad\quad \times
        \left(
            \Vert
                [
                    G(\theta_0, \theta(\sigma_1))^{-1}
                    - G(\theta_0, \theta(\sigma_2))^{-1}
                ] w_2
            \Vert_{E_{w, 1, p, q, T}}
            + \Vert
                \sigma_1
                - \sigma_2
            \Vert_{E_{\sigma, 1, p, T}}
        \right)\\
        & \leq C T^{1 - 1/p}
        \Vert
              \sigma_1
            - \sigma_2
        \Vert_{H^{1, p}(0,T)}
        (
            1
            + \Vert
               w_1
            \Vert_{E_{w, 1, p, q, T}}^{2l}
            + \Vert
                \sigma_1
            \Vert_{E_{\sigma, 1, p, T}}^{2l}
        )\\
        & \quad\quad\quad\quad \times
        (
            \Vert
                w_1
            \Vert_{E_{w, 1, p, q, T}}^2
            + \Vert
                \sigma_1
            \Vert_{E_{\sigma, 1, p, T}}^2
        )\\
        & + C
        \sum_{j = 1, 2}
        (
            \delta(T)
            + \Vert
                w_j
            \Vert_{E_{w, 1, p, q, T}}^{2l}
            + \Vert
                \sigma_j
            \Vert_{E_{\sigma, 1, p, T}}^{2l}
        )\\
        & \quad\quad\quad\quad \times
        \left(
            \Vert
                \sigma_1
            \Vert_{H^{1, p}(0,T)}
            \Vert
                w_1
                - w_2
            \Vert_{E_{w, 1, p, q, T}}
            + \Vert
                \sigma_1
                - \sigma_2
            \Vert_{E_{\sigma, 1, p, T}}
        \right)\\
        & + C
        \sum_{j = 1, 2}
        (
            \delta(T)
            + \Vert
                w_j
            \Vert_{E_{w, 1, p, q, T}}^{2l}
            + \Vert
                \sigma_j
            \Vert_{E_{\sigma, 1, p, T}}^{2l}
        )\\
        & \quad\quad\quad\quad \times
        \left(
            \Vert
                \sigma_1
                - \sigma_2
            \Vert_{H^{1, p}(0, T)}
            \Vert
                w_2
            \Vert_{E_{w, 1, p, q, T}}
            + \Vert
                \sigma_1
                - \sigma_2
            \Vert_{E_{\sigma, 1, p, T}}
        \right)\\
        & \leq C
        \sum_{j = 1, 2}
        (
             T^{1 - 1/p}
            + \delta(T)
            + \Vert
                w_j
            \Vert_{E_{w, 1, p, q, T}}^{2l + 2}
            + \Vert
                \sigma_j
            \Vert_{E_{\sigma, 1, p, T}}^{2l + 2}
        )\\
        & \quad\quad\quad\quad \times
        \left(
            \Vert
                w_1
                - w_2
            \Vert_{E_{w, 1, p, q, T}}
            + \Vert
                \sigma_1
                - \sigma_2
            \Vert_{E_{\sigma, 1, p, T}}
        \right).
    \end{align*}
    We continue to split as
    \begin{align*}
        & F_w^3(w_1, \sigma_1)
        - F_w^3(w_2, \sigma_2)\\
        %& = (\theta(\sigma_1) - \theta_0) R_2(\theta_0) G(\theta_0, \theta(\sigma_1))^{-1} w_1
        %- (\theta(\sigma_2) - \theta_0) R_2(\theta_0) G(\theta_0, \theta(\sigma_2))^{-1} w_2\\
        & =
        \left[
            \theta(\sigma_1)
            - \theta(\sigma_2)
        \right] R_2(\theta_0) G(\theta_0, \theta(\sigma_1))^{-1} w_1\\
        & + (
            \theta(\sigma_2) - \theta_0
        ) R_2(\theta_0)
        \left[
            G(\theta_0, \theta(\sigma_1))^{-1}
            - G(\theta_0, \theta(\sigma_2))^{-1}
        \right] w_1\\
        & + (
            \theta(\sigma_2) - \theta_0
        ) R_2(\theta_0) G(\theta_0, \theta(\sigma_2))^{-1}
        \left(
            w_1
            - w_2
        \right).
    \end{align*}
    We find from Corollary \ref{cor_bound_of_G_in_E_w1pqT} that
    \begin{align*}
        & \Vert
            F_w^3(w_1, \sigma_1)
            - F_w^3(w_2, \sigma_2)
        \Vert_{E_{w, 0, p, q, T}}\\
        & \leq C T^{1 - 1/p}
        \Vert
            \sigma_1
            - \sigma_2
        \Vert_{H^{1, p}(0, T)}
        (
            1
            + \Vert
                \sigma_1
            \Vert_{H^{1, p}(0, T)}
        )
        \Vert
            w_1
        \Vert_{E_{w, 1, p, q, T}}\\
        & + C T^{1 - 1/p}
        \Vert
            \sigma_2
        \Vert_{H^{1, p}(0, T)}
        (
            1
            + \Vert
                \sigma_1
            \Vert_{H^{1, p}(0, T)}
        )
        \Vert
            \sigma_1
            - \sigma_2
        \Vert_{H^{1, p}(0, T)}
        \Vert
            w_1
        \Vert_{E_{w, 1, p, q, T}}\\
        & + C T^{1 - 1/p}
        \Vert
            \sigma_2
        \Vert_{H^{1, p}(0, T)}
        (
            1
            + \Vert
                \sigma_1
            \Vert_{H^{1, p}(0, T)}
        )
        \Vert
            w_1
            - w_2
        \Vert_{E_{w, 1, p, q, T}}\\
        & \leq C T^{1 - 1/p}
        (
            1
            + \sum_{j = 1, 2} \Vert
                w_j
            \Vert_{E_{w, 1, p, q, T}}^3
            + \Vert
                \sigma_j
            \Vert_{E_{\sigma, 1, p, T}}^3
        )\\
        & \quad\quad\quad\quad \times
        (
            \Vert
                w_1
                - w_2
            \Vert_{E_{w, 1, p, q, T}}
            + \Vert
                \sigma_1
                - \sigma_2
            \Vert_{E_{\sigma, 1, p, T}}
        )
    \end{align*}
    We finally find from Corollary \ref{cor_bound_of_G_in_E_w1pqT} that
    \begin{align*}
        & \Vert
            F_\sigma (w_1, \sigma_1)
            - F_\sigma (w_2, \sigma_2)
        \Vert_{L^q(0, T)}\\
        & \leq C
        (
            \delta(T)
            + \sum_{j = 1, 2}
            \Vert
                G(\theta_0, \theta(\sigma_j))^{-1} w_j
            \Vert_{E_{w, 1, p, q, T}}^m
            + \Vert
                \sigma_j
            \Vert_{E_{\sigma, 1, p, T}}^m
        )\\
        & \quad\quad\quad \times
        (
            \Vert
                G(\theta_0, \theta(\sigma_1))^{-1} w_1
                - G(\theta_0, \theta(\sigma_2))^{-1} w_2
            \Vert_{E_{w, 1, p, q, T}}
            + \Vert
                \sigma_1
                - \sigma_2
            \Vert_{E_{\sigma, 1, p, T}}
        )\\
        & \leq C
        (
            \delta(T)
            + \sum_{j = 1, 2}
            (
                1
                + \Vert
                    \sigma_j
                \Vert_{H^{1, p}(0, T)}
            )^m
            \Vert
                w_j
            \Vert_{E_{w, 1, p, q, T}}^m
            + \Vert
                \sigma_j
            \Vert_{E_{\sigma, 1, p, T}}^m
        )\\
        & \quad\quad\quad \times
        \Biggl (
            \Vert
                \sigma_1
                - \sigma_2
            \Vert_{E_{w, 1, p, q, T}}
            (
                1
                + \sum_{j = 1, 2}
                \Vert
                    \sigma_j
                \Vert_{E_{w, 1, p, q, T}}
            )
            \Vert
                w_1
            \Vert_{E_{w, 1, p, q, T}}\\
        & \quad\quad\quad\quad\quad\quad
            + (
                1
                + \sum_{j = 1, 2}
                \Vert
                    \sigma_j
                \Vert_{E_{w, 1, p, q, T}}
            )
            \Vert
                w_1
                - w_2
            \Vert_{E_{w, 1, p, q, T}}
            + \Vert
                \sigma_1
                - \sigma_2
            \Vert_{E_{\sigma, 1, p, T}}
        \Biggr )\\
        & \leq C
        \left(
            \delta(T)
            + \sum_{j = 1, 2}
                \Vert
                    w_j
                \Vert_{E_{w, 1, p, q, T}}^{2m + 2}
                + \Vert
                    \sigma_j
                \Vert_{E_{\sigma, 1, p, T}}^{2m + 2}
        \right)\\
        & \quad\quad\quad\quad\quad\quad\quad\times
        (
            \Vert
                w_1
                - w_2
            \Vert_{E_{w, 1, p, q, T}}
            + \Vert
                \sigma_1
                - \sigma_2
            \Vert_{E_{\sigma, 1, p, T}}
        ).
    \end{align*}
    By applying these estimates to Theorem \ref{thm_well_posedness_nonlinear_propblem_abstract}, we obtain the desired conclusion
\end{proof}
\begin{proof}[Proof of Theorem \ref{thm_main}]
    Theorem \ref{thm_main} follows from the definitions of $\tilde{v}$ and $v_c$ (see \eqref{eq_definition_of_tilde_v_vc} and \eqref{eq_parabolic_eq_for_tilde_v}) together with Lemma \ref{lem_nonlinear_problem}.
\end{proof}

%%===========================================================================================%%
%% If you are submitting to one of the Nature Portfolio journals, using the eJP submission   %%
%% system, please include the references within the manuscript file itself. You may do this  %%
%% by copying the reference list from your .bbl file, paste it into the main manuscript .tex %%
%% file, and delete the associated \verb+\bibliography+ commands.                            %%
%%===========================================================================================%%

%\bibliography{sn-bibliography}% common bib file
%% if required, the content of .bbl file can be included here once bbl is generated
%%\input sn-article.bbl
\section*{Acknowledgments}
The author was partly supported by JSPS Grant-in-Aid for Young Scientists No. JP22K13948.

\begin{appendices}

\end{appendices}

\end{document}